\documentclass[10pt]{amsart}

\usepackage[
    style=numeric,
    natbib=false,
    url=false, 
    doi=false,,
    eprint=false,
    isbn=false
]{biblatex}
\renewbibmacro{in:}{}
\usepackage{amsfonts,xr,graphicx,amsmath, amssymb}
\usepackage{enumerate} 

\newtheorem{theorem}{Theorem}
\newtheorem{prop}{Proposition}[section] 
\newtheorem{theorem2}[prop]{Theorem}
\newtheorem{lemma}[prop]{Lemma}
\newtheorem{lem}[prop]{Lemma}
\newtheorem{rem}[prop]{Remark}
\newtheorem{defn}[prop]{Definition} 
\newtheorem{cor}[prop]{Corollary}

\DeclareMathOperator{\rank}{rank}

\DeclareMathOperator{\id}{1} 
\DeclareMathOperator{\re}{Re} 
\DeclareMathOperator{\Div}{div}

\begin{document}
\begin{abstract} We study partial analyticity of solutions to elliptic
  systems and analyticity of level sets of solutions to nonlinear
  elliptic systems.  We consider several applications, including
  analyticity of flow lines for bounded stationary solutions to the
  2-d Euler equation, and analyticity of water waves with and without
  surface tension. 
\end{abstract} 
\thanks{The authors acknowledge support by the DFG through SFB 1060 and the Alexander von Humboldt foundation for the second author.} 

\title[Partial analyticity and nodal sets] {{Partial analyticity and nodal sets for nonlinear elliptic systems} }
 \author{Herbert Koch }
\address{Mathematisches Institut \\ Universit\"at Bonn \\ Endenicher Allee 60 \\ 53115 Bonn \\ Germany }
\email{koch@math.uni-bonn.de} 
\author{ Nikolai Nadirashvili}
\address{Aix-Marseille Universit\'e, \\ I2M \\ 39 rue F. Joliot-Curie \\ 13453
Marseille \\ France,}

\email{nikolay.nadirashvili@univ-amu.fr}

\date{}

\maketitle

\def\C{\mathbb{C}}
\def\S{\mathbb{S}}
\def\Z{\mathbb{Z}}
\def\R{\mathbb{R}}
\def\N{\mathbb{N}}
\def\H{\mathbb{H}}
\def\T{\mathbb{T}}
\def\tilde{\widetilde}
\def\epsilon{\varepsilon}

\def\n{\hfill\break} \def\al{\alpha} \def\be{\beta} \def\ga{\gamma} \def\Ga{\Gamma}
\def\om{\omega} \def\Om{\Omega} \def\ka{\kappa} \def\lm{\lambda} \def\Lm{\Lambda}
\def\dd{\delta} \def\Dl{\Delta} \def\vph{\varphi} \def\vep{\varepsilon} \def\th{\theta}
\def\Th{\Theta} \def\vth{\vartheta} \def\sg{\sigma} \def\Sg{\Sigma}
\def\bendproof{$\hfill \blacksquare$} \def\wendproof{$\hfill \square$}
\def\holim{\mathop{\rm holim}} \def\span{{\rm span}} \def\mod{{\rm mod}}
\def\rank{{\rm rank}} \def\bsl{{\backslash}}
\def\il{\int\limits} \def\pt{{\partial}} \def\lra{{\longrightarrow}}
\def\pa{\partial } 
\def\ra{\rightarrow }
\def\sm{\setminus }
\def\ss{\subset }
\def\ee{\epsilon }

\section{Introduction}

For nonlinear elliptic equations and elliptic systems given by 
functions  depending  analytically on their  arguments any sufficiently regular solution defined in an open domain is necessarily a real-analytic function.
This classical result was established in the pioneering works of
S. Bernstein \cite{MR1544961}, H. Levy \cite{MR1512557}, I. Petrovsky 
\cite{MR0001425} and Morrey \cite{MR2492985}. Our goal is to generalize
these results to nonlinear elliptic systems which are analytic with
respect to a group of spatial variables: we assume that the dependence
of the equations on spatial variables is real-analytic only for a part
of variables. Under such an assumption we prove that the solutions are
analytic in the same group of variables. We call that property the
partial analyticity.

We are interested in analytic solutions to nonlinear elliptic systems
\begin{equation}\label{nonlinear}   \sum_{j} \partial_j F^j_k(x,u_i,Du) = f_k(x,u,Du)\qquad 1\le k \le N,  \end{equation}  
where $u:D\to \R^N$, $D$ is a bounded domain in $\R^n$, $F_k$ and
$f_k$ are real.  We assume that $F^j_k$ and $f_k$ are defined on an
open set
\[ U \subset \R^n  \times \R^N \times \R^{n\times N}   \] 
and we define  
\[ a^{ij}_{kl}(x,u,p) = \partial_{p_j^l} F_k^i(x,u,p) \] where
$(x,u,p)$ denotes a typical element of the domain of definition, which
may sometimes be ambiguous since we also use $u$ as a notation for a
function $u:D \to \R^N$.

\begin{defn} We call the coefficients $(a^{ij}_{kl}) \in \R^{n \times n \times N \times N}$ elliptic if for all $0\ne \xi \in \R^n$ the matrix 
\[  A_{kl}(\xi)  = \sum_{i,j=1}^n a^{ij}_{kl} \xi_i \xi_j  \] 
is invertible. We call \eqref{nonlinear} elliptic if there exists $C$ so that
\begin{enumerate} 
\item The coefficients  are bounded: $|a^{ij}_{kl}(x,u,p)|\le C$.
\item The inverses are bounded: 
$\Vert (A_{kl}(x,u,\xi))^{-1} \Vert \le C |\xi|^{-2} $. \end{enumerate} 
\end{defn}    

Let $n_1+n_2=n$, $1\le n_1,n_2 <n$ and denote $x=(x',x'') \in \R^{n_1}
\times \R^{n_2}$. The main object of interest is analyticity with
respect to $x'$. There are different equivalent definitions of
analyticity: Bounds on the growth of derivatives, locally uniform
convergence of power series, and extension to a holomorphic function
on a complex domain. All this is qualitative, whereas proofs often 
rely on quantitative formulations of analyticity. 

With this notation we formulate the first main result where we
consider a small perturbation of a constant coefficient system. This
can be achieved by localizing and rescaling for partially analytic
elliptic systems with $C^1$ regularity. Trivially every affine  function $u: \R^n \to \R^N$
satisfies 
\[ \sum_{i,j,l} \partial_i \left(a^{ij}_{kl} \partial_j u^l\right) = 0 \] 
for all $k$. We consider solutions as perturbations of affine functions. 

\begin{theorem}\label{theorem1} 
Let $0<s<1$, $(a^{ij}_{kl})$ be elliptic,  $\bar u \in \R^N$, $b \in \R^{n\times N}$. There exist $\delta >0$, $\rho>0$ and $\varepsilon$   depending only on $n$ 
and   $a^{ij}_{kl}$  
such that the following is true. If 
\begin{enumerate}[(i)] 
\item  Let 
\[ V \subset (\C^{n_1}\times \R^{n_2})  \times \C^N  \times \C^{n\times N} \]
contain a   $ 2 \varepsilon$ neighborhood of  
\[   \left\{ \left(x'+iy',x'', \bar u + b \left( \begin{matrix} x'+iy'\\ x''\end{matrix}\right),b \right), x=(x',x'')  \in \overline{B_1(0)} , |y'| \le \delta (1-|x|^2)_+^3 \right\}.   \]       
We assume that $F^j_k,f_k : V \to \C$ are bounded measurable
  functions, holomorphic in $z'=x'+iy'$, $u$  and $p$ for almost all $x''$.
\item The inequality  
\[   |F^i_k(z',x'',u,p)-F^i_k(\tilde z',\tilde x'',\tilde u,\tilde p)- a^{ij}_{kl} (p-\tilde p) | \le  \rho(|x''-\tilde x''|^s + |u-\tilde u|+|z-\tilde z|)  + \varepsilon |p-\tilde p|  \]
holds in $V$.  Moreover 
\[       \Vert f_k(z',x'',u,p)  \Vert_{sup} + 
  \Vert D_{u,p} f_k \Vert_{sup}      < \rho. \] 
\item  The weak solution $u \in C^{0,1} (B_2(0),\R^N)$ to \eqref{nonlinear} satisfies  
\[   \Vert u(x) - (\bar u + b \cdot x) \Vert_{C^{0,1}(B_2(0))} \le \varepsilon. \]  
\end{enumerate} 
Then  the solution $u$ has a unique extension $u^{\C}$ to 
\[  B_1^\C:= \Big\{ (x'+iy',x''): |(x',x'')| \le 1 , |y|\le \delta (1-|x|^2)_+^3 \Big\}\]   such that  
\[  \Vert u^{\mathbb{C}}-(\bar u + b(x+iy'))   \Vert_{\sup,B_1^{\mathbb{C}}} 
+ \Vert Du^{\mathbb{C}}-b    \Vert_{\sup,B_1^{\mathbb{C}}}
 < 2 \varepsilon \]  
and $u$ is holomorphic with respect to $x'+iy'$ for all $x''$. 
\end{theorem} 

\begin{rem} By the Cauchy integral this implies estimates 
\begin{equation}    |\partial^\alpha D_x u(x)|  \le   c |\alpha|! R^{-|\alpha|}    \end{equation}  
for every multiindex $\alpha$ corresponding to $\R^{n_1}$
\[  R =  \delta  (1-|x|^2)_+^3 /2. \]
We call $u$ uniformly partially analytic if $c$ and $R$ can be chosen uniformly on compact sets. 
\end{rem}

The smallness condition can be achieved if $u, F$ and $f$ are
continuously differentiable at any point $x$ by choosing a small ball
around $x$ and rescaling it to size $1$. This will be done in detail 
in the proof of Theorem \ref{theorem2}. There is not much we know about the size of
$\delta$ even in simple situations: Consider the homogeneous Laplace
equation. Then we may choose any $\delta <1$, but we do not know whether 
this is optimal, nor whether more is known.

If $n_2=N=1$   much weaker  smallness conditions are sufficient for a similar conclusion.

\begin{theorem}\label{theorem1a} 
Let  $n \ge 2$, $(a^{ij}(x_n))$ be a measurable one parameter family of
bounded and uniformly elliptic coefficients, $b \in \R^{n-1}$, $\bar u\in C^{0,1}((-2,2))$ and $J\subset I \subset \R$
intervals, where $J$ is compact and $I$ is open.  Then there exist
 $\delta>0$, $\varepsilon>0$ and $\rho>0$ depending only on $n$, the bound and the ellipticity constant of the coefficients  such that the
following is true.  
\begin{enumerate}[(i)] 
 \item Suppose that the open set 
\[ V \subset (\C^{n_1}\times \R)  \times \C  \times \mathbb{C}^{n} \]
contains  a $2 \varepsilon$ neighborhood of (with $I$ corresponding to $\partial_n u$)  
\[   \left\{ \left(x'+iy',x_n,  \bar u(x_n) + b \cdot (x'+iy'), b
\right) : |x| \le 1 , 
|y'| \le \delta (1-|x|^2)_+^3  \right\} \times I . 
\]  
 We assume that $F^j,f : V \to \C$ are bounded measurable 
  functions, holomorphic in $x'$, $u$  and $p$ for almost all $x_n$.
\item The inequality  
\[ 
  |F^i(z',x_n,u,p)-F^i(z',x_n,\tilde u,\tilde
  p)- a^{ij}(x_n) (p_j -\tilde p_j) | \le 
  \rho |u-\tilde u|
  + \varepsilon |p-\tilde p| 
 \]
holds for points in $V$ and 
\[  |f| + |D_{u,p}f|   \le \rho     \]
\item  Let  $u \in C^{0,1} (B_2(0))$ satisfy
\[    \Big\Vert u(x) - (\bar u(x_n)+ \sum_{j=1}^{n-1}  b_j x_j)\Big\Vert_{C^{0,1}(B_1)} \le \varepsilon  \]
\[  \frac{\bar u(x_n)-\bar u(\tilde x_n)}{x_n-\tilde x_n}   \in J \]  
and, in a distributional sense,  
\[ \partial_i F^i(x,u,Du) = f(x,u,Du). \] 
\end{enumerate} 
Then  $u$ has a unique extension $u^{\C}$ to $B_1^{\C}$
 such that
\[  \Vert u^{\C}-(\bar u(x_n) +  \sum_{j=1}^{n-1} b_j(x_j+iy_j) )   \Big\Vert_{sup,B_1^{\mathbb{C}}}
+  \Vert Du^{\C}-  (b, \partial_n \bar u )   \Big\Vert_{sup,B_1^{\mathbb{C}}}
   < 2 \varepsilon \]  
and $u$ is holomorphic with respect to $x'+iy'$ for almost all $x''$. 
\end{theorem}

\begin{rem} The analogous result holds for Dirichlet boundary
  conditions and for elliptic boundary value problems for a boundary given by $x_n=0$. The proof
  covers diagonal systems and general coercive equations if $n=2$.
  Note however the following example.  The function
\[ u(x_1,x_2)= \arctan(x_1/x_2) \] 
is harmonic in $x_2>0$, it satisfies 
\[ u(x_1,0) = \left\{\begin{array}{ll}  \pi/2 & \text{ if } x_1 >0 \\ -\pi/2 & \text{ if } x_1 <0 \end{array} \right. =: h(x_1).  \] 
The derivative 
\[ \partial_{x_2} u(x_1,x_2)=   \frac{x_1}{x_1^2+x_2^2} \] 
is unbounded. In particular it is not uniformly analytic 
 in $x_2$ near  $x_2=x_1=0$ for fixed $x_1$.  A small modification gives a 
counterexample for homogeneous Dirichlet data. Let 
\[  v(x_1,x_2) = \int_0^{x_2} \arctan(x_1/t) dt \] 
It satisfies $v(x_1,0)=0$ and
\[
\begin{split} 
 \Delta  v =&  \int_0^{x_2} \partial_{x_1}^2 \arctan(x_1/t) dt + 
\partial_{x_2} \arctan(x_1/x_2) 
\\ = & \int_0^{x_2} \Delta  \arctan(x_1/t) dt+  h(x_1), 
\end{split} 
\] 
  hence 
\[ \Delta v = h(x_1), \qquad v(x_1,0)=0. \]
Clearly $\partial_{x_2}^2 v$ is not uniformly bounded and hence $v$ is
not uniformly analytic in $x_2$ near $x_1=x_2=0$.  Notice that in the classical case of complete
analyticity of the equation solutions are analytic up to the boundary
\cite{MR2492985}.  The example shows that for partial analyticity it
is important that the boundary is tangential to the analytic
direction.
\end{rem} 

The main tool of the proof of Theorem \ref{theorem1} and Theorem
\ref{theorem1a} is a complexification of the equation in a group of
variables in which the analyticity holds. The classical approach of
Levy, Petrovsky and Morrey is also based on the complexification of
the equations, however their methods essentially used the whole set of
variables. Our argument applies to any number of variables and it gives
an independent proof of analyticity for analytic elliptic boundary
value problems.

\begin{cor} Suppose that $F$ and $f$ are analytic in all arguments, 
$u\in C^{1}$ satisfies 
the elliptic equation 
\[ \partial_i F^i (x,u,Du) = f(x,u,Du). \]
Then $u$ is analytic. The same result holds for analytic elliptic
 boundary value  problems.
 \end{cor} 

\begin{proof} The first statement follows from Theorem \ref{theorem1} 
with $n_2=0$, possibly after localizing to a small neighborhood of a point and rescaling. We did require $n_2>0$, but there is no change of the proof 
for $n_2=0$. The argument for analyticity in tangential directions for 
boundary value problems is the same as for partial analyticity, with the 
difference that the notion of ellipticity is harder to formulate and we 
refer to \cite{MR0162050} and \cite{MR2492985} for a discussion. 

The theorem of Cauchy-Kowalevskaja \cite{kow} implies existence of an analytic 
solution for analytic Cauchy data. In particular it gives the correct bounds 
of any derivative  in terms of  the analytic Cauchy data. We use these bounds at each level $x_n$ assuming that the boundary is given by $x_n=0$. 
\end{proof}

Theorems \ref{theorem1} and \ref{theorem1a} have important consequences
for level sets of solutions $u$. Roughly speaking the regularity
statement does not change if we do a coordinate change in dependent
and independent variables, which interchanges the role of partial
analyticity and analyticity of level sets.

\begin{theorem}\label{theorem2}  
  Let $n=n_1+n_2$, $N=n_2$, $0<s<1$, $U \subset \R^n$ be open and let $ u \in C^{1}(U,\R^N)$ be a weak solution to the equation 
\[ \partial_i F^i_k(x,u,Du) = f_k(x,u,Du) . \]
We assume that 
\begin{enumerate}[(i)]  
\item  The  open set $ V \subset \R^n  \times \R^N \times \R^{n\times N}$ contains the closure of  
\[  \{ (x,u(x), Du(x)) :  x \in U \}. \]
The functions $F^i_k(x,u,p)$ are continuous and the functions $f_k(x,u,p)$ are 
measurable  and  uniformly analytic with respect to $x$ and $p$, more precisely  there exists $R$ and $c$ such that 
\[\sup_{\alpha}  R^{|\alpha|}  \left\Vert   \partial_{x,p}^\alpha F^{i}_k  \right\Vert_{C^s}   \le c \qquad \text{ and }  \sup_{x,u,p,\alpha}  R^{|\alpha|}   \sup_{x,p} | \partial_{x,p}^\alpha f_k(x,.,p) |  \le c. \] 
\item  The coefficients
\[    a^{ij}_{kl}(x) = \frac{\partial F^i_k}{\partial p^j_l} \left(x,u(x), Du(x)\right) \] 
are elliptic.
\item  $Du$ has rank $N$ in $U$. 
\end{enumerate} 
Then the level sets 
\[   \{ x \in U | u(x) = y\}  \] 
are uniformly analytic for all $y$ on compact sets. Moreover  $Du$ is uniformly analytic when restricted to  a compact subset of a level set: i.e. if $W \subset \R^{n_1}$ open, $\phi : W \to \R^{n_2}$ is an analytic  function whose graph is in $U$ such that 
 $ x'\to u(x',\phi(x'))$ is a constant function then 
\[   x'\to Du(x',\phi(x')) \] 
is analytic uniformly in the level.   
\end{theorem} 

The continuity assumption can be relaxed if $N=1$. 

\begin{theorem}\label{theorem2a} 
Let $ u \in C^{0,1}(U)$ be a weak solution to 
\[ \partial_i F^i(x,u,Du)= f(x,u,Du). \]
Suppose that 
\begin{enumerate}[(i)] 
\item  The  set $ V \subset \R^n  \times \R \times \R^n$ contains the closure of  
\[  \{ (x,u(x), Du(x)) :  x \in U \} \]
and, if $(x,u,p) \in V$ and $(x,u,q) \in V$ then also $(x,u,p+t (q-p)) \in V$ 
for $0\le t \le 1$.  The functions $F$ and $f$ are uniformly analytic in $V$ with respect to $x$ and $p$  in the sense that there exists $C,R> 0$ such that 
\[  R^{|\alpha|}\left(  \Vert \partial^\alpha_{x,p} F^i(x,u,p) \Vert_{sup} 
+  \Vert  |\partial_{x,p}^\alpha f|  \Vert_{sup} \right) 
\le c. \] 
\item The coefficients 
\[ a^{ij}(x)= \frac{\partial F^i}{\partial p^j} (x,u(x), Du(x)) \]  
are uniformly positive definite. 
\item  $u$ is uniformly monotone in the direction $0 \neq v \in \R^n$,
i.e.
\[  \inf_{x,h}  \frac{u(x+h v) -u(x)}{h} > 0.  \] 
\end{enumerate} 
Then all level sets of $u$ in $U$ are uniformly 
analytic. Moreover for almost all 
levels $u$ is differentiable at this level set, and the restriction of 
$Du$ to the level set is uniformly analytic on compact sets.  
\end{theorem}

In Section \ref{sec:examples} we consider applications of Theorem
\ref{theorem2} to problems from different areas of analysis.  Even
though the theorems are quite general and flexible we consider the
techniques to be applicable in many more situations.  This is
illustrated in some of the examples where we rely on modifications of
the proofs instead of  an application of the theorems. 

The proof of the Theorems follows in the remaining half of the paper:
Singular integral estimates in Section \ref{sec:singular} state
standard results and a crucial variant needed for Theorem
\ref{theorem1a} and Theorem \ref{theorem2a}. Partial analyticity for
solutions to linear elliptic equations to the theme of Section
\ref{sec:linan}. The corresponding result for nonlinear equations
follows by a standard fixed point argument in Section
\ref{sec:nonlinan}, where Theorem \ref{theorem1} and Theorem
\ref{theorem1a} are proven. The paper is completed with Section
\ref{sec:transform} where Theorem \ref{theorem2} and Theorem
\ref{theorem2a} are deduced from Theorem \ref{theorem1} and Theorem
\ref{theorem1a} via a change of coordinates - with the important
feature that it involves dependent and independent coordiantes.

The theorems of this introduction are the main abstract results and we
number them by single digits. The remaining numbering is done counting
within the section. We use the summation convention, but without a
consistent use of upper and lower indices.

 \section{Examples}
 \label{sec:examples} 
 In this section we discuss several examples arising from different
 fields of analysis and physics.
 
\subsection{Lagrangian trajectories of fluid dynamics}
 
Smoothness of streamlines (trajectories of material particles of the
fluid) is a classical subject. For many hydrodynamical models
solutions in Eulerian variables (velocity, pressure) are less regular
than in Lagrangian variables.
 
The study of regularity of streamlines goes back to the work of
Lichtenstein \cite{MR1544732}, Chemin and Serfati. For historical surveys of the problem
see Majda and Bertozzi \cite{MR1867882}. With a maximal regularity  approach the regularity of streamlines was settled in Serfati \cite{MR1325824} and again taken up in Zheligovsky and Frisch \cite{MR3216577} and  
Constantin, Vicol and Wu \cite{constantin2014analyticity}. The authors
proved that for the Euler equation of ideal fluid, for the
quasi-geostrophic equation and for the Boussinesq equation that if  the
classical solutions are defined on a time interval and have Eulerian
velocities in $C^{1,\epsilon},\, \epsilon >0$ then their Lagrangian
paths are real analytic.
 
On the opposite extreme, existence of flow lines has been proven by Bianchini and Gusev \cite{arXiv:1408.2932}  under very weak assumptions on the velocity vector field.

We discuss in more detail the Euler equation of an ideal fluid in
dimension 2 where we will obtain a related result under considerably  weaker
regularity assumptions.  Let
\[ v(t,x)= (v_1(t,x),v_2(t,x)), \quad  t\in \R, x\in \R^2  \] 
 be  the velocity and $p(t,x)$ the pressure of a solution of the Euler equation for an ideal fluid: 

\begin{equation}\label{equ:euler} \left\{
\begin{array}{l l}
\pa v/\pa t +v \nabla v = -\nabla p, &\mbox{in $\R^2\times \R $} \\
\Div  v=0 &\mbox{in  $ \R^2 \times \R$} \\
\end{array} \right. \end{equation}
For any initial data $v_0\in C_0^{k,s}(\R^2) $, $k=1,2,..., 0<s<1$
\[v(x, 0)= v_0(x)  \] 
there exists a unique
solution $v(t,x) $ of (1) defined for all $t\in \R $ such that for all $t\in \R $ $v\in C^{k, s}(\Omega)  $, see \cite{MR1637634,MR1867882, MR1923695}.

The vector field $v(x,t)$ defines a flow $z_x(t)$ on $\R^2$,
which  is a one-parametric group of area preserving  diffeomorphisms of $\R^2 $. 
Let $x_0\in \R^2$. The curve $z_{x_0}(t)\in \R^2$, 
$$\gamma :t\in\R\ra z_{x_0}(t)\in\R^2 $$
is called the streamline of a material particle with initial position $x_0$ of the fluid.

From the results of Serfati   \cite{MR1325824} 
it follows that if the initial velocity $v_0$ is in $ C^{1,\epsilon},\, \epsilon >0$
then for any $z_0$  the map $\gamma$ is real analytic. 
  The assumption that $v\in C^{1,\epsilon},\, \epsilon >0$,
is essential.  The following is an 
instance  of Theorem \ref{theorem2} where  the assumptions on the
regularity of $v$ are considerably weakened   for stationary flows.

\begin{theorem2}\label{statEuler1}   Suppose that $v$ is a  bounded stationary
  solution to the 2-d Euler equation \eqref{equ:euler} 
 on the unit disc $D$ and that there exists $\delta >0$ with $v_1 > \delta$.
Then there exists $\mu=\mu( \Vert v \Vert_{L^\infty}/ \delta) $ and a Lipschitz map  
\[  \psi:  \Big\{ (x_1+iy_1,x_2)\in \C \times \R:  |x_1|^2+ |y_1/\mu|^2+ x_2^2 \le \frac14 \Big\} \to \C.    \] 
 It is uniformly holomorphic in $z_1$ for  $(z_1,x_2)$ in compact sets,
  $\psi(0,x_2)=x_2$
and for  almost all $x_2$ and all $x_1$ 
\begin{equation} \label{tangential}    v_1(x_1,\psi(x_1,x_2)) \psi_{x_1}(x_1,x_2) -v_2(x_1, \psi(x_1,x_2)) = 0.  \end{equation} 
Moreover for almost all $x_2$ the map 
\[  x_1 \to v(x_1,\psi(x_1,x_2)) \] 
is analytic, again uniformly in compact sets ( for $x_1$ and $x_2$).  
 \end{theorem2} 

The identity \eqref{tangential} implies that the vector field is
tangential to the graph of $ x_1 \to \psi(x_1,x_2)$ and that there is
an analytic flow line.  The simplest such flows are shear flows $v=
(v_1(x_2),0)$, with $v_1 $ bounded from below and above. The flow
lines are horizontal curves, and the restriction of the velocity to a
flow line is defined for almost all $x_2$, and it is constant and
hence analytic whenever it is defined.

This theorem is vague about what we require from $v$ to call it a
solution to the stationary Euler equation, and this is an important
point we want to clarify. The simplest definition would be a distributional 
solution to 
\begin{equation}\label{equ:euler2} \left\{
\begin{array}{l l}
 \sum_{i=1}^2 \partial_i (v_i  v_j)  = -\partial_j p, &\mbox{in $\R^2\times \R $} \\
\Div  v=0 &\mbox{in  $ \R^2 \times \R$} 
\end{array}\right.  \end{equation}
but for technical reasons we are not able to work with this definition. There is
a formulation which may be more relevant for the physical interpretation. 
Let 
\[ e = p + \frac12 |v|^2 \] 
be the inner energy. Bernoulli's law states that it is constant along flow 
lines. We calculate using Euler's equation
\begin{equation}\label{inneren} 
\begin{split} 
 \nabla e = & \nabla p + \left( \begin{matrix} 
  v_1 \partial_1 v_1 + v_2 \partial_1 v_2 \\
  v_1 \partial_2 v_1 + v_2 \partial_1 v_1 \end{matrix} \right) \\
  = & (\partial_1 v_2 - \partial_2 v_1) \left( \begin{matrix} v_2 \\ -v_1 \end{matrix} \right). 
\end{split} 
\end{equation} 
The last vector is perpendicular to the velocity vector field and this implies 
Bernoulli's law. The last identity can be rewritten in divergence form 
\begin{equation}\label{weakeuler} 
 \begin{split} 
  \partial_1 (\frac12( v_1^2-v_2^2)) + \partial_2 (v_1 v_2) + \partial_1 e =& 0\\ 
  \partial_2 (\frac12( v_2^2-v_1^2)) + \partial_1 (v_1 v_2) + \partial_2 e =& 0\\
   \nabla \cdot v = & 0 
\end{split} 
\end{equation} 
Since the velocity vector field is divergence free 
there exists a velocity potential $\Phi$
with $v = \left( \begin{matrix} \partial_2 \Phi \\ -\partial_1
    \Phi \end{matrix} \right) $, since $\nabla \cdot v=0$, which is
Lipschitz continuous if the velocity vector field is bounded. 
Then the system reads as  
\begin{equation}\label{stateuler2}  
\begin{split} 
\frac12 \partial_1  (\Phi_2^2-\Phi_1^2)- \partial_2 ( \Phi_1 \Phi_2) + \partial_1 e =& 0\\
\frac12 \partial_2 (\Phi_1^2- \Phi_2^1) - \partial_1 (\Phi_1 \Phi_2) +  \partial_2 e  = & 0   
\end{split}  
\end{equation} 
combined with 
\[  \Delta e = \frac12 \left( \partial_{11}^2 (\Phi_1^2-\Phi_2^2) + \partial_{22}^2(\Phi_2^2-\Phi_1^2) \right) + \partial^2_{12} (\Phi_1 \Phi_2). \] 

This formulation is related to the one used 
by Delort \cite{MR1102579}  and Evans and M\"uller \cite{MR1220787}. 
Up to this point we did only use that $\Phi \in H^1_{loc}$. Now suppose that 
$u \in L^\infty \cap W^{1,1}$ has integrable derivatives. Then by \eqref{inneren} 
 $e$  has integrable derivatives and  $\nabla e $ and $\nabla \Phi$
are linearly dependent. If in addition 
$\partial_2 \Phi > \delta >0$ then $e$ is a function 
of $\Phi$, $e= e(\Phi)$ for some function $e$ and $\Phi$ satisfies the elliptic equation 
\begin{equation}\label{stat-euler}    
\partial_2( \frac12( \Phi_1^2-\Phi_2^2) )-\partial_1 (\Phi_1 \Phi_2) + \partial_2 e(\Phi) =0 
\end{equation} 
which, under the assumption that $\Phi$ has integrable second derivatives becomes
\[ \Delta \Phi = e'(\Phi) \] 
 for an integrable function $e'$. 
By a  weak solution in  Theorem \ref{statEuler1} we mean  a weak solution to 
  \eqref{stat-euler} for some local integrable function $\Phi \to e$.
If $(v,e) $ is a weak solution to \eqref{weakeuler},  if $\Phi$ is the velocity potential and if $v_1 > \kappa >0$ then Bernoulli's law may be formulated as 
the requirement that $e$ is a function of $\Phi$ which implies 
\eqref{stat-euler}.

The claim of the Theorem \ref{statEuler1} is a consequence of Theorem
\ref{theorem2a} without assuming anything on $e$ beyond local
integrability. Only boundedness of the inner energy $e$ has to be deduced from
the upper and lower bounds on the velocity.
 
Suppose that  $\Phi\in C^{0,1}$ and $e$ be functions which satisfy \eqref{stat-euler} and 
$\Phi_2 > \delta $. Then we obtain a bound on the oscillation of $e$ in terms
of $\delta$ and the Lipschitz constant.  We choose a non-negative test function $h \in C^\infty_0(-1/2,1/2)$ 
and integrate $h(x_1)$ over $t_1\le \Phi \le t_2$. 
We apply  the divergence theorem  to \eqref{stat-euler} multiplied by $h$:  
\[
\begin{split} 
 \left. \frac12 \int_{\Phi=t}  \frac{1}2 |D\Phi| \Phi_2         h(x_1) d\mathcal{H}^1 \right|_{t=t_1}^{t=t_2} -\int_{t_1<\Phi< t_2}
      \Phi_1\Phi_2      \partial_{x_1}  h(x_1)  dx  & 
\\ & \hspace{-5cm} = \left.\int_{\Phi=t} h(x_1) 
  \frac{\Phi_2}{\sqrt{\Phi_1^2+ \Phi_2^2}} d\mathcal{H}^1 e(t)
\right|^{t=t_2}_{t=t_1} 
\end{split} 
\] 
and, after adding a suitable constant we
obtain a bound on $e$, $\Vert e \Vert_{sup} \le C$ with a constant
depending only on the Lipschitz bound on $D\Phi$ and $\delta$.

\begin{cor} {\bf (The maximum principle for the argument of a stationary flow) } Suppose that $v \in C(\Omega)$  is a continuous  stationary
  solution to the 2-d Euler equation \eqref{stat-euler} 
 on a simply-connected domain $\Omega$ such that $v\neq0$ in $\Omega$. Let $x_0\in \Omega$. Let
 $arg\, v$ be the (continuous) argument of the vector field $v$ defined such that $arg\, v(x_0)=0$. Then
 $$
 \inf_{\partial \Omega} arg\, v\le 0 \le \sup_{\partial \Omega} arg\, v
 $$
with equality if and only if the direction of $v$ is constant. 
\end{cor} 

\begin{proof} The proof anticipates the main reduction for Theorem
  \ref{theorem2a}.  It suffices to verify the claim in a small ball. 
Then, without loss of generality after a rotation there is a lower bound 
on the first component of the velocity vector. Let $\Phi$ be the velocity potential.  We introduce new coordinates
  $(y_1,y_2)=(x_1,\Phi(x))$ and express $x_2=u(y)$ as a function of
  $y$. Then 
\[ \frac{\partial y}{\partial x} = \left( \begin{matrix} 
  1 & 0 \\ \Phi_{x_1} & \Phi_{x_2} \end{matrix} \right) \qquad  
\text{ and }\qquad 
 \frac{\partial x}{\partial y} = \left( \begin{matrix} 
  1 & 0 \\  u_{y_1}  & u_{y_2}  \end{matrix} \right). \]
Hence 
\[   u_{y_2} = \Phi_{x_2}^{-1} , \quad u_{y_2} = - \frac{\Phi_{x_1}}{\Phi_{x_2}}, \]
\[   \frac12 (\Phi_{x_2}^2 - \Phi_{x_1}^2) = 
\frac12 \Big( \frac{1- u_{y_1}^2}{u_{y_2}^2} \Big) \] 
and 
\[  \Phi_{x_1} \Phi_{x_2} = - \frac{ u_{y_1}}{u_{y_2}^2}. \] 
The Jacobian determinant of $x\to y$ is $\Phi_{x_2}= (u_{y_2})^{-1} $. Thus, for suitable test functions $\varphi$,  
\[ 
\begin{split} 
  \int \frac12 (\Phi_{x_2}^2 - \Phi_{x_1}^2) \varphi_{x_2} + \Phi_{x_1} \Phi_{x_2} \varphi_{x_1}   dx 
 \hspace{-2cm} & \\  = & 
\int \frac12\Big( \frac{1-u_{y_1}^2}{u_{y_2}^2}
\frac1{u_{y_2}}  \varphi_{y_2} 
   - \frac{u_{y_1}}{u_{y_2}^2} (\varphi_{y_1} - \frac{u_{y_1}}{u_{y_2}} \varphi_{y_2}     \Big) u_{y_2}    dy 
\\ & = \int \frac12 \frac{1+u_{y_1}^2}{u_{y_2}^2} \varphi_{y_2} 
- \frac{u_{y_1}}{u_{y_2}} \varphi_{y_1} dy 
\end{split} 
\]   
Similarly 
\[ \int e(\Phi) \varphi_{x_2} dx = \int e(y_2)  \varphi_{y_2} dy \] 
and $u$ is a weak solution to 
\begin{equation} \label{eulertransformed}  \partial_{y_1} \Big(\frac{u_{y_1}}{u_{y_2}}\Big) -  \frac12\partial_{y_2} \Big( \frac{1+ u_{y_1}^2}{u_{y_2}^2}+2e(y_2) \Big)    = 0  \end{equation}
if and only if $\Phi$ satisfies \eqref{stat-euler}. 
  
We differentiate  equation \eqref{eulertransformed} 
with respect to $y_1$ and denote $v= u_{y_1}$. 
Then 
\[ \partial_i \left(a^{ij} \partial_j v\right)  = 0 \] 
where 
\[   (a^{ij}) = \left( \begin{matrix} \frac{1}{u_{y_2}} & -  \frac{u_{y_1}}{u_{y_2}^2} \\ -\frac{u_{y_1}}{u_{y_2}^2} &  \frac{1+u_{y_1}^2}{u_{y_2}^3} 
\end{matrix} \right) \] 
It is  positive   definite since it is symmetric, the  entry $a^{11}$   is positive   and its determinant is  $u_{y_2}^{-4}$. As a consequence 
\[   v = - \frac{\Phi_{x_1}}{\Phi_{x_2}} \] 
satisfies a maximum principle, and it assumes its maximum and its minimum at the boundary. The argument of the velocity vector field (up to an additive constant)  is 
\[  \arctan\Big( \frac{\Phi_{x_1}}{\Phi_{x_2}} \Big) \] 
and the claim follows since $\arctan$ is strictly monotone.
Of course this argument is not rigorous since it requires regularity of $v$, 
but it can easily be replaced by an argument using finite differences.  
\end{proof} 

We also have 
\begin{lemma} Suppose that $\Phi\in C^{0,1}(B_1)$ satisfies $\partial_2  \Phi> \kappa >0$  and \eqref{stat-euler}. Then $(\partial_2
  \Phi,-\partial_1 \Phi)$  is a stationary solution to 
the  Euler equation \eqref{equ:euler}. Moreover there is a sequence 
$(u^j,p^j) $ of smooth solutions to \eqref{equ:euler} with 
\[  \Vert u^j \Vert_{L^\infty} \le 2 \left\Vert  \left(\begin{matrix} \partial_2 \Phi \\  - \partial_1 \Phi \end{matrix} \right) \right\Vert_{L^\infty}, \]  
\[ \inf u^j_1 \ge \delta /2 \] 
and for any compact subset $K \subset B_1(0)$ 
\[   \left\Vert u^j - \left(\begin{matrix} \partial_2 \Phi \\ - \partial_1 \Phi \end{matrix} \right) \right\Vert_{L^2(K)}   \to 0. \]  
\end{lemma} 

\begin{proof} We transform the problem by the change of dependent and
  independent variables as above.  For the transformed problem we can
  regularize $F$, see the proof of Theorem \ref{theorem1a}. \end{proof}

The authors do not know whether every  bounded weak stationary solution
to the Euler equations \eqref{equ:euler2} with a lower bound on $v_1$ 
satisfies Bernoulli's law in our formulation.   The set of such
solutions is however closed under weak* convergence in $L^\infty$, but the
question of approximability seems to be difficult, and we seem to be
missing tools to address this delicate question.

\subsection{Traveling water waves} 

\label{freeboundary} 

The presentation  here is  related to the work of Constantin and Escher
\cite{MR2753609}, who proved the first result in this direction assuming more 
regularity and required the flow to be partially periodic, with no flux 
boundary conditions at the bottom.

Our result applies to all flow lines, not only the free
boundary. Moreover we consider local in space solutions in contrast to
the global assumption in \cite{MR2753609}, and our results also apply
to capillarity waves i.e. those with surface tension.  The interior
case has been discussed in the last section.

We consider a traveling gravity-capillarity wave in two dimensions in
a moving frame of velocity $c$ in the $x_1$ direction. The stationary Euler equations with a constant vertical
gravitational field in a moving frame are
\begin{equation} 
\begin{split} 
   (v_1-c)\partial_{x_1} v_1  + v_2 \partial_{x_2} v_1  = - \partial_{x_1} p 
\qquad & (v_1-c) \partial_{x_1} v_2 + v_2 \partial_{x_2}  v_2 = -\partial_{x_2}  p-g\\  \partial_{x_1} v_1 + \partial_{x_2} v_2 = 0 \qquad & \Delta p = -\partial^2_{x_1} v_1^2 - \partial^2_{x_2} v_2^2 -2 \partial^2_{x_1x_2} (v_1 v_2)
\end{split} 
\end{equation} 
in a domain $U$. We assume that  the boundary 
contains (this part  we call upper boundary and free boundary)  
the graph 
of a function $\eta$:  We assume that there is an open interval $V$ such that 
\[ U \subset \left\{ (x_1,x_2) \in \R^2 :  x_2 < \eta(x_1) , x_1 \in V \right\}.
\]
The boundary conditions at the free boundary (the graph of $\eta$) are the momentum  balance  
\begin{equation} \label{boundpressure}   p = \kappa \mathcal{H} = -\kappa \partial_{x_1} \Big( \frac{\eta_{x_1}}{\sqrt{1+\eta_{x_1}^2}} \Big) \end{equation} 
where $\kappa $ is the surface tension and where  $\mathcal{H}$ is the mean curvature, and  the dynamic boundary condition 
\begin{equation}\label{bounddynamic}    v_2 = (v_1-c) \eta_x \end{equation} 
which ensures that flow lines are tangential at the free boundary.  

Since the flow is incompressible there exists a stream function $\phi$
so that $v_2= -\partial_{x_1} \Phi$ and $v_1 = \partial_{x_2} \Phi+c $.
The stationary Euler equations can be rewritten in terms of $\Phi$ as in 
the previous section,

\begin{equation} 
\begin{split}  \partial_{x_1} \Phi_2^2 - \partial_{x_2} (\Phi_1  \Phi_2)  +\partial{x_1} p =&  0 \\  
  -\partial_{x_1} (\Phi_2 \Phi_1)  + \partial_{x_2} (\Phi_1^2)   +\partial_{x_2} p =& -g  
\end{split} 
\end{equation} 
We define the energy density   
\[  e=  \frac12 (| \Phi_1|^2+|\Phi_2|^2) + p +g x_2 \] 
and  the stationary Euler equations become
\begin{equation} 
\begin{split}   \frac12 \partial_{x_1}\Big(\Phi_2^2-\Phi_1^2\Big) + \partial_{x_2} (\Phi_1\Phi_2)  = & \partial_{x_1} e  \\  
  \partial_{x_1} (\Phi_2 \Phi_1)  + \frac12 \partial_{x_2} \Big(\Phi_1^2-\Phi_2^2\Big)   = & \partial_{x_2} e.   
\end{split} 
\end{equation} 

By the dynamic boundary condition $\Phi$ is constant on the free boundary. It is determined up to a constant which we choose so that $\Phi=0$ at the free boundary. Again $e$ is a function of $\Phi$ and it is constant at the free boundary.
We denote this constant by $e_0$. 

The dynamic boundary condition becomes $ \Phi= 0 $ 
and the momentum balance is
\begin{equation}   \frac12 |\nabla \phi|^2 + g \eta = \kappa \partial_{x_1} \left( \frac{ \partial_{x_1} \eta}{ \sqrt{1+ (\partial_{x_1} \eta)^2 }}\right)-e_0.  \end{equation}

We assume $ v_1 < c$, i.e. a wave is faster than the supremum of the
velocity of the particles. This implies that $\partial_{x_2}\phi  \le -\kappa
<0$. Then the velocity potential $\Phi$ satisfies again \eqref{stat-euler}
in a set where the free boundary is given by $x_2= \eta(x_1)$.
The boundary condition is 
\[   e- g x_2 - \frac12 |\nabla \Phi|^2= \kappa \mathcal{H}. \] 
We  rewrite the problem as 

\begin{equation} \label{water wave formulation} 
\begin{split} 
\frac12 \partial_{x_2} (\Phi_2^2-\Phi_1^2-2e(\Phi) ) + \partial_{x_1} (\Phi_1 \Phi_2) = 0    & \text{ in }  x_2 < \eta(x_1) \\
\frac12 |\nabla \Phi|^2  + g\eta   =e+ \kappa \partial_{x_1} \Big(\frac{\eta_{x_1}}{\sqrt{1+\eta_{x_1}^2}}\Big)  & \text{ on }  x_2 = \eta(x_1) \\
\Phi=0 & \text{ on } x_2 = \eta(x_1) . 
\end{split} 
\end{equation}

We introduce new coordinates $(y_1,y_2)=(x_1,\phi(x))$ and express $x_2=u$ as a
function of $y$. Then we obtain the same equation  as above for the bulk 
\[ 
 \partial_{y_1} \Big(\frac{u_{y_1}}{u_{y_2}}\Big) - \frac12\partial_{y_2}\Big( \frac{1+ u_{y_1}^2}{u_{y_2}^2}\Big) = \partial_{y_2} e(y_2) , 
\] 
combined with the boundary condition
\begin{equation}    \frac12  \frac{1+u_{y_1}^2}{u_{y_2}^2}+e_0 = g u -\kappa 
\partial_{y_1} \Big(\frac{u_{y_1}}{\sqrt{1+u_{y_1}^2}}\Big)   \label{capillary} 
\end{equation} 

\begin{theorem2}\label{waterwave} Let $g,\kappa \in \R$. 
Suppose that the velocity field $\Phi$ is the stream function of a bounded solution which satisfies  
  \eqref{stat-euler} in an open set $U$ as above together with the
  boundary conditions \eqref{boundpressure} and \eqref{bounddynamic}.
  We assume that the horizontal velocity $c+\partial_{x_2} \Phi $ is
  below $c$.  Then the flow lines and the free boundary are analytic.
\end{theorem2}

\begin{proof} 
Let $w= \partial_{y_1} u$. It satisfies formally  
\[ \begin{split} \partial_{i}a^{ij}\partial_j w  = & 0 \qquad \text{ for } y_2 >0 \\ 
  a^{2j} \partial_j w - g w = & -\partial_{x_1} (  b    \partial_{x_1} w ) \qquad \text{ for } y_2 = 0 
\end{split} 
\] 
where $a^{ij}$ is the positive definite matrix 
\[ \left( \begin{matrix}  - \frac1{u_{y_2}} &  \frac{u_{y_1}}{u_{y_2}^2}  \\  \frac{u_{y_1}}{u_{y_2}^2} & - \frac{1+u_{y_1}^2}{ u_{y_2}^3} \end{matrix} \right) \]  
and 
\[ b = \kappa \Big(1+(\frac{u_{y_1}}{u_{y_2}})^2\Big)^{-3/2}. \] If $\kappa=0$
then the assertion follows from the boundary version with Neumann
boundary conditions of the statements of the last section.  We
postpone the proof for the case $\kappa\ne 0$ to the end of the paper.
\end{proof} 

\subsection{The obstacle problem} 

We consider viscosity solutions to 
\[   \min\{ u , 1-\Delta u \} =0 \] 
in an open set which is a concise formulation of the obstacle problem.
Alternative formulations are 
\[ u \ge 0 , \Delta u \ge 1 , u (\Delta u -1 ) = 1 \] 
as well as the variational characterization as local minimizers of 
\[ \int \frac12 |\nabla u |^2 + u  dx \] 
subject to the constraint $u\ge 0$. 
 Solutions have locally Lipschitz continuous derivatives. The contact set 
 is defined to be $K=\{ x: u(x)=0\}$. Its boundary is the free boundary. The regular 
part consists of points $x_0$ in the free boundary 
such that in a neighborhood outside $K$ possibly after a rotation 
\[ \partial^2_{x_nx_n} u \ge \kappa > 0. \]
On $K$ one has $\nabla u = 0 $. 

\begin{theorem} The regular part of the free boundary is analytic. 
\end{theorem}

This is well known due to the work of Caffarelli and coworkers \cite{MR803243} who prove that  $u\in C^{2,\alpha}$ on the closure of the positivity set
near a regular free boundary point. The regular free boundary
is analytic    by the work of Kinderlehrer and Nirenberg \cite{MR0440187}. 
   Here we want to show that the obstacle problem  fits into 
the context of this paper. 

\begin{proof} 

Locally we may assume that the free boundary is the graph of a Lipschitz 
function $x_n= \eta(x')$, and that the contact set is below the graph. 
Let $v= \partial_n u $. Above the graph it is harmonic. It vanishes 
at the boundary.

Moreover, if $\phi \in C^\infty_0$ with sufficiently small support, then 
\begin{equation}  
\label{weak} 
\begin{split} \int \nabla  v \nabla  \phi dx = & -\int \nabla u \partial_n \nabla \phi dx
\\   = & \int_{x_n> \eta}   \partial_n \phi dx  
\\ = & -\int_{\R^{n-1}}  \phi(x',\eta(x'))  dx'
\end{split} 
\end{equation}    
by an application of Fubini's theorem. 

Then $v $ is harmonic in the positivity set. It is Lipschitz 
continuous and the level surfaces are Lipschitz graphs. It satisfies 
\[ v = 0 \] 
and the conormal boundary condition
\[   \partial_{\nu}  v = -1   \] 
at the boundary where  $\nu$ denotes the exterior normal. 

Now let $y_i=x_i$, $y_n=v$ and $w(y)=x_n$. We change coordinates  
in equation \eqref{weak}: 
\[
\begin{split} 
 \int \nabla_x v \cdot \nabla_x \varphi dx 
= & \int \Big[ \sum_{i=1}^{n-1} 
(- \frac{w_{y_i}}{w_{y_n}} ( \varphi_{y_i} - \frac{w_{y_i}}{w_{y_n}} \varphi_{y_n} ) + \frac{1}{w_{y_n}^2}  \varphi_{y_n} \Big] w_{y_n} dy 
 \\
= & \int -\sum_{i=1}^{n-1} w_{y_i} \varphi_{y_i} + \frac{1+ \sum_{i=1}^{n-1} w_{y_i}^2}{w_{y_n}} \varphi_{y_n} dy  
\end{split} 
\] 
Thus $w$ is a weak solution to 
\[   \sum_{i=1}^{n-1}  \partial_i^2 w  
- \partial_n \frac{1+ \sum_{j=1}^{n-1} |\partial_j w|^2}{\partial_n w} 
= 0 \] 
with boundary condition 
\[   \frac{1+ \sum_{j=1}^{n-1} |\partial_j w|^2}{\partial_n w}    = 1.    \] 
This is a boundary analogue of Theorem \ref{theorem2a} and it can be
proven along the same lines. Alternatively we may write 
\[  w = x_n+ \hat w, \] 
check that 
\[ \sum_{i=1}^{n-1} \partial_i^2 \hat w + \partial_n   \frac{\partial_n \hat w -\sum_{j=1}^{n-1} \partial_j \hat w^2}{1+ \partial_n \hat w}    =   0 \] 
\[   \frac{\partial_n \hat w - \sum_{j=1}^{n-1} |\partial_j \hat w|^2}{1+\partial_n w}    = 0.    \] 
An even extension of $\hat w$ is a weak solution of an elliptic equation of the 
type considered in Theorem \ref{theorem2a} with a discontinuity at $y_n=0$ 
and the claim follows from the proof pf Theorem \ref{theorem2a} below.   
\end{proof}

\subsection{Harmonic maps }
 
 Let $(M^n,g)\, (N^k,h)$ be Riemannian manifolds with metrics $\sum g_{lm}dx^ldx^m$ and
 $\sum h_{ij}du^idu^j$. Let $u$ be a $C^1$ map from $M^n$ to $N^k$.
A  map $u:M^n\to N^k $ is called
 harmonic map if its Dirichlet integral is a local  extremal of the map.

 The preimage
 $u^{-1}(\{y\})$ is  called the level set to the level $y$. As a consequence of Theorem \ref{theorem2}  we have
 
\begin{theorem2} 
  Assume $(M^n,g)$ is a real analytic Riemannian manifold, $(N^k,\delta)$  is a $C^{2}$ manifold with a $C^1$ metric $\delta$ and $k<n$. Then all level sets of the continuous harmonic map $u$ outside its
 critical points are  real-analytic submanifolds. 
\end{theorem2}

\begin{proof} 
It is well known that continuous harmonic maps are as smooth as the data permit, 
which is here $C^{1,s}$ for all $0<s<1$, see \cite{MR0244627} . 
Let $x\in M^n$, $y=u(x) \in N^k$. We choose local coordinates 
near $x$ in $M$, denote the Laplace-Beltrami operator of $M$ by $\Delta_M$, 
and denote the Christoffel symbols of the Levi-Civita connection on $N$ 
by $\Gamma^{ij}_l$. The Christoffel symbols are continuous. 
In these coordinates the harmonic map equations i.e. the Euler-Lagrange equations of the 
Dirichlet integral can be written as 
\[  \Delta_M u_k = \sum_{\alpha, i,j} \Gamma^{ij}_k(u)  \partial_\alpha u_i \partial_\alpha u_j.
 \] 
The statement in local coordinates follows now from Theorem \ref{theorem1}. 
\end{proof}

\subsection{Nonlinear Schr\"odinger equations}
 
 Nonlinear Schr\"odinger equations of the type
 \begin{equation} \label{nls} 
i\frac{\partial u} {\partial t}+  \Delta u =  f (|u|)u  ,     
\end{equation} 
where $u:\Omega \to \C,\, \Omega \ss \R^n,\, p\geq 0$. 
This  appears in many
physical models: internal gravity waves, ferromagnetism, nonlinear
optics and plasma theory, see \cite{MR1932182}. 
For  $f=|u|^2$  and '$+$'
\eqref{nls} is the Ginsburg-Landau equation.
It also  appears in nonlinear optics.  Stationary solutions of \eqref{nls} as
well traveling wave type solutions satisfy an elliptic system
\eqref{nlsstat}.
  The nodal set $u^{-1}(\{0\})$ is  of  particular interest.
 
 We do not require that $f$  in
 \eqref{nls} is  analytic. However, as a consequence of Theorem
 2 we see that nodal surfaces of stationary solutions to \eqref{nls}
 are real analytic.
 
\begin{theorem2} 
\label{theorem4} 
Let $u $ be a bounded complex valued stationary solution of \eqref{nls} in a domain $\Omega
\ss \R^n$:
\begin{equation} \label{nlsstat} \Delta u + f(|u|)u=0  \end{equation} 
with some continuous function $f$.
Let $\Gamma $ be a nodal set of $u$, $z\in \Gamma$
and real $rank \, Du(z)=2$. Then in a neighborhood of $z$ $\Gamma$ is a 
$n-2$-dimensional real analytic surface. If $D    |u|^2 \ne 0$ and $r>0$ then 
\[ \{x : |u(x)|=r\}  \] 
is a real analytic surface of dimension $n-1$. 
\end{theorem2}  

\begin{proof} The first statement is an immediate consequence of Theorem 
\ref{theorem2}. For the second statement we have to check the proof. First
  standard elliptic estimates imply that the function $u$ is in $C^1$.  
We fix a point $x_0$ with $|u(x_0)|=r$. Without loss of generality 
we may assume that $u(x_0)=r$.  Let 
\[  w(x) = \Im u(x) \]
Without loss of generality we assume that $\partial_{x_n} |u|^2(x_0) \ne 0$. 
We choose 
\[ y_i = x_i \qquad \text{ if } 1\le i < n \qquad y_n = |u|, \qquad v(y)= x_n.  \] 
A tedious calculation leads to the  elliptic  system of equations for $v$ and $w$
and an application of Theorem \ref{theorem1} implies the claim.
We leave the tedious and instructive calculation to the reader.

\end{proof}

\subsection{Free boundary problems of Grad-Mercier type} 

We consider a free boundary problem of the type arising from
variational inequalities.  Let $\Omega \ss \R^n$ be a bounded domain, and
$u\in W^{2,p}(\Omega )$, satisfying,
 
\begin{equation}\label{X} \left\{
\begin{array}{l l}
-\Delta u+g(u)= 0, &\mbox{in }\Omega \\
u=\; \mbox{(unknown) constant }>0&\mbox{on } \partial \Omega \\
\int_{\partial \Omega}{\frac{\partial u}{\partial n}}=  -1  
\end{array} \right. \end{equation}
where  $\Gamma =\{ x\in \Omega , u(x)=0\}$ is a free boundary. Equation \eqref{X} is
an equivalent form of the Grad-Mercier equation of equilibrium of
plasma, see \cite{MR558592}. The function $g$ is a bounded but
generically discontinuous function. At $0$ function $g$ has a discontinuity of the first kind. A typical example,
$g(t)=t$ for $t>0$ and $g(t)>C>0$ for $t<0$, see \cite{MR558592}. 
  
  As a consequence of Theorem \ref{theorem2a} we have
  
\begin{theorem2}
Suppose that $0$ is a point in the free boundary 
$  \Gamma $ and $\nabla u(0)\neq 0$. 
   Then $\Gamma$ is an analytic  surface in a neighborhood of $0$.
\end{theorem2}

Assume additionally in the theorem that $n=2,\, \Omega$ is a convex
domain and $u$ is a locally stable solution of a corresponding
variational functional. Under such assumptions Cabre and Chanillo 
\cite{MR1623694} proved that a solution $u$ of a variational problem
has a single critical point. Thus as a corollary we have: under the
above additional conditions to Theorem \ref{theorem2a} the free
boundary $\Gamma$ is a closed analytic curve.

\subsection{Examples and counter examples for $C^2$ regularity}

Let $u$ be a solution of a semi-linear equation
\begin{equation}\label{sl}
\Delta u=f(u)
\end{equation}
in a domain $U\subset \R^n$. Suppose $f\in C^{s}$, $0\le s \le 1$. For $0<s<1$ the inner regularity of $u$ follows from the standard Schauder estimates:
\[ 
u\in C^{2+s}
\] 
For $s=0$ standard estimates show that $u$ lies locally in the Besov  $B^{2}_{\infty,\infty}$.  However, we show for some classes of solutions of equation \eqref{sl} for $s =0$ the implication $u\in C^{2}$ holds true.

Let $u\in B^2_{\infty,\infty}(D)$ and $\nabla u(x_0)=0,\, x_0\in D$. We say that $x_0$
is a Morse point if $u$ has a  second differential at $x_0$ of
the Morse type, i.e., there exists a quadratic form $q(x)$ with
non-zero eigenvalues such that
\[ 
q(x)-u(x-x_0)=o(|x|^2).
\] 

\begin{lemma}\label{Ck} 
Let $u$ be a solution of \eqref{sl}. Suppose $f\in C^0$. If all critical points of  $u$ are of Morse type then $u\in C^ 2$.  
\end{lemma}

\begin{proof} 
Critical points of Morse type are isolated. 
 Let $B\subset D$ be a ball in $D$ and $|u|>\delta >0$ in $B$.
We assume that $\nabla u \ne 0 $ in $B$ and 
set $e_1(x)=\nabla u/|\nabla u|$
\[ 
\Gamma_t=\{ x\in B,\, u(x)=t\}.
\] 
Since by assumption $\nabla u \ne 0$ the level surfaces are regular, and they depend continuously on $t$. 
By Theorem \ref{theorem2a} the level sets $\Gamma_t$ are locally uniformly
analytic. Let $H(x)$ be the mean curvature of the level set $\Gamma_{u(x)}$
at the point $x$. It is a continuous function of $x$. Now
\[\Delta u =  \sum_{i,j=1}^n e_{1,i} e_{1,j} \partial_{ij}^²u + H \partial_{e_1}
u  = f(u) \] 
and, since all terms besides the
first are continuous, the first term is continuous as well, and all
second derivatives are continuous.

It remains to   prove the theorem in neighborhoods of the critical points of the function $u$.
Suppose that $0\in U$ is a critical point of $u$. Define $u_a=u(ax)/a^2,\, a>0$. From the assumptions
of the theorem
$$
u_a\to U
$$
in $C^1(B_1)$ as $a\to 0$, where $B_1\subset \R^n$ is a unit ball and $U$ is a quadratic form with non-zero eigenvalues.  Therefore $\nabla u_a$ is bounded away from $0$ in the spherical shell $B_1\setminus B_{1/2}$
for all sufficiently small  $a$. Thus $|D^2u_a|$ is bounded in $B_1\setminus B_{1/2}$ for all small $a$ and $u_a \to  U$ in $C^1$ and 
$$
D^2u_a\to D^2U
$$
weak* in $L^\infty$. 

The level sets are uniformly analytic, and the same is true for the gradient 
restricted to the level sets. As a  consequence second order derivatives 
containing one tangential direction converge uniformly. Since we can solve the equation for the remaining second order derivative
$D^2u_a \to D^2 U$ uniformly in $B_1\backslash B_{1/2}$.  

\end{proof}

In  \cite{preprint:shahgholian}  H. Shahgholian raised a question: Is it true that a function $u\in C^1$ which satisfies  \eqref{sl} with a continuous non-linearity $f$ is automatically a $C^2$ function?  Below we show that in Theorem \ref{Ck} one can not drop the assumptions on the critical points of the solution.

\begin{lemma} There exists a continuous function $f$ 
with values in $[-1,0]$ 
and $ u \in B^2_{\infty,\infty}(B_1(0))$ which satisfy 
\begin{equation}  \Delta u = f(u) \qquad \text{ in } B_1(0) \end{equation}  
and $u(x) = x_1x_2 $ for $|x|=1$ which has no bound of the type $C |x|^2$.   
More precisely 
$u(0) = Du(0) = 0 $ and 
\[    \sup_{x}  \frac{ u(x)}{|x|^2} = \infty. \]  
\end{lemma} 

\begin{proof} 
Let $Q$ be the first quadrant of $\R^2$:
\[ 
Q=\{ x\in \R^2: x_1>0, x_2>0\}. 
\] 
Let $w$ be a solution in $Q$ of the Dirichlet problem
\begin{equation}\label{De,Q} \left\{
\begin{array}{l l}
\Delta w= -1, &\mbox{in }Q\cap B_1(0) \\
w=0 &\mbox{on } \partial Q  \cap B_1(0) \\
w=  x_1 x_2 & \mbox{on } Q \cap \partial B_1(0) 
\end{array} \right. \end{equation}
One easily checks that 
\[  \Delta \{[(-\ln(x_1^2+x_2^2))+1] x_1x_2\}  = -2 \frac{x_1x_2}{|x|^2} \]  
and that this function  satisfies the boundary conditions. The solution $u$ to 
\[ \Delta \tilde u  = -1 + \frac{2x_1x_2}{x_1^2+x_2^2} \] 
with boundary data $\tilde w(x) = 0$ if  $|x|=1$ or $x_1=0$ or $x_2=0$
is given by 
\[ 
\tilde u(x) = \frac{4}{\pi} \sum_{j=2}^\infty \frac{1}{(2j)^3} \Big(
   \frac{\re  (x_1+ix_2)^{2j}}{|x|^{2j-2}} -\re  (x_1+ix_2)^{2j} \Big)    
\] 
which is easily seen to  be
twice differentiable. Thus 
\begin{equation} w - [(-\ln(x_1^2+x_2^2))+1] x_1x_2 \in C^2_b \end{equation} 
has bounded second order derivatives.

We will choose functions $f$ with values between $0$ and $1$. By the
maximum principle any solution (which will in general be non unique) is
bounded from below by the positive harmonic function $x_1x_2$, and from  above
by $w$.  For any sequence of functions $f$
converging to the negative of the Heaviside function the solutions
converge to $w$. Hence, for $k \in \mathbb{N}$ there exists
$f_k$ such that
\[ \sup \frac{ u_k}{|x|^2} \ge k^2. \] 
We define 
\[   f(u) = \max \frac{f_k(u)}k. \] 
The solution $u$ satisfies 
\[  \limsup_{x\to 0}  \frac{u}{|x|^2}   \ge \sup_{k} \limsup_{x\to 0} 
\frac{ u_k(x)}{k|x|^2} = \infty.  \]

Suppose that the function $f$ is extended oddly on $\R$:
$f(t)=-f(-t)$. We chose an odd extension of the function $u$ over the
coordinate axis on $\R^2$. The extended $u$ will satisfy the equation
$\Delta u=f(u)$ on $B_1$. The function $u$ is continuously
differentiable with $\nabla u(0)=0 $. Due to the lower bound $u$ cannot
be twice differentiable at $0$.
\end{proof}

\section{Singular integral type estimates} 
\label{sec:singular} 
We consider a linear elliptic system  
\[ \partial_ia^{ij}_{kl} \partial_j u^l = \partial_i F^i_k \]
with $1\le i,j \le n$ and $1\le k,l \le N$ 
which we write in divergence form. The map $F \to \nabla u$ is a Calder\'on-Zygmund operator
and the following estimates are standard. 

\begin{prop}\label{calderon} 
Suppose that $u$ has derivatives in $L^1 \cup L^q$ for some $q<\infty$
and 
\[ \partial_ia^{ij}_{kl} \partial_j u^l = \partial_i F^i_k .\] 
Then 
\[ \Vert D u \Vert_{L^p} \le c_p \Vert F \Vert_{L^p}. \]
Moreover, if the derivatives of $u$ grow at most  linearly and $0<s < 1$ then 
\[ \sup_{x\ne y}  \frac{| D u(x) - D u(y)|}{|x-y|^s} 
\le c \sup_{x\ne y}    \frac{| F(x) - F(y)|}{|x-y|^s} .  \]
\end{prop}

For $n_2=1$ there is a variant to the first H\"older estimate: We may restrict
 to difference quotients in the first $n-1$ variables. We recall the notation 
$(x',x'') \in \R^{n-1} \times \R$.

\begin{prop}\label{prop:strange} 
Suppose that 
\[ \partial_i a^{ij}_{kl} \partial_j u^l = \partial_i F^i_k \] 
in $\R^n$. If 
\[ |\nabla u | \le c (1+|x|) \] 
then 
\[ \sup_{x'\ne y',x''} \frac{| D u(x',x'') - D u(y',x'')|}{|x'-y'|^s} 
\le c \sup_{x'\ne y',x''}    \frac{| F(x',x'') - F(y',x'')|}{|x'-y'|^s}  \]
 \end{prop}

 \begin{proof}  We formulate the crucial result 
for singular integral operators.  
\begin{lemma}
\label{convkernel} Let $0<s<1$. 
 We denote a point in $\R^n$ by $(t,x)\in \R \times \R^{n-1}$. Let $C^s_0$ be the homogeneous H\"older space of H\"older continuous functions with compact support and let 
\[ T : C^s_0 (\R^n) \to  L^\infty   \] 
be a partial convolution  operator with integral kernel  $k(\tau,t,x)$  
(i.e. 
\[ Tf(t,x)= \int k(\tau,t,x-y) f(\tau,y) dy d\tau \]
 under suitable  assumptions on the support)   which satisfies 
\[ |k(\tau,t,x)| \le \frac{c}{|(\tau-t,x)|^{n}},  \]
\[ |k(\tau,t,x)-k(\tau,t,y)| \le   c \frac{|x-y|}{\min\{ |(\tau-t,x)|,|(\tau-t,y)|\}^{n+2}}, \]
and, if $t\ne \tau$ and $R>0$
\begin{equation} \label{cancelation}  \left| \int_{B_R^{\R^{n_1}}(0)} k(\tau,t,x)\,  dx \right|  \le  c  R^{-1} \end{equation} 
Then $T$ has a unique extension to $T:L^\infty_t C^s_x \to L^\infty_t  C^s_x$  , for $0<s<1$ and it satisfies  
\[ \sup_{t,x\ne y} 
\frac {| Tf(t,x)-Tf(t,y) |}{|x-y|^s}  \le c 
\sup_{t,x\ne y}  \frac{ | f(t,x)-f(t,y) |}{|x-y|^s} \] 
\end{lemma} 

\begin{rem} We denote the seminorm by $C^s_*$, 
\[ \Vert f \Vert_{C^s_*} = \sup_{t,x\ne y}  \frac{ | f(t,x)-f(t,y) |}{|x-y|^s}.\]
\end{rem}

We apply this to the Calder\'on-Zygmund operator
\[ f \to \partial^2_{ij} u. \] 
As a homogeneous convolution operator of Calder\'on-Zygmund operator type
with smooth kernel it  always satisfies the first two conditions. The cancellation condition follows from the fact that the kernel is the derivative in
one of the directions, unless $i=j=n$, but in that case we see that it
is the sum of second derivatives by using the equation (the kernel is
a solution to the equation with respect to the first variable, and to
the adjoint equation with respect to the second variable away from the diagonal).
Proposition \ref{prop:strange} follows from 
Lemma \ref{convkernel} which we turn to next.  \end{proof}

\begin{proof}[Proof of Lemma \ref{convkernel}]

Let  $x_1\in \R^{n-1}$ with $|x_1|=1$. The claim follows from the estimate
\begin{equation}\label{mainstrange} 
  \left| \int_{-\infty}^\infty \int_{\R^{n-1}} (k(\tau,0,x_1-y)-k(\tau,0,-y)) f(\tau,y) dx \,  
d\tau \right| \le c \Vert f \Vert_{L^\infty_t \dot C^s_x} 
\end{equation}  
by translation in $t$, $x$  and scaling. It is part of the assertion that 
the integral exists. We
turn to the proof of \eqref{mainstrange}.
In a first step we show that it suffices to prove this 
estimate under the additional assumption $f(t,0)=0$.

\noindent{\bf Step 1:} 
We fix $\eta \in C^\infty_0(\R^{n-1})$, supported in the unit ball $B_1 \subset \R^ {n-1}$ 
and identically $1$ in a ball of radius $1/2$. Then by the cancellation 
condition and the pointwise condition on the kernel
\begin{equation} \label{cancelation2} 
\begin{split} 
 \left|\int k(\tau,0,x)  \eta(x/R)  dx \right| = & \left| \int_{B_{R/2}} k(\tau,0,x) dx\right|  + \left| \int_{B_R \backslash B_{R/2} } k(\tau,0,x) \eta (x) dx \right| 
\\ \le & C                     R^{-1}. 
\end{split} 
\end{equation} 
 
Let $f \in L^\infty_t \dot C^s$ (where $\dot C^s$ is the homogeneous H\"older space with semi norm $\sup \frac{|u(x)-u(y)|}{|x-y|^s} $) 
with compact support. We define for $\rho\ge 1$
\[   R(t) =  \rho (1+ |f(t,0)|)(1+|t|^2) \] 
and 
\[  g(t,x) = f(t,0) \eta(x/R(t)). \]  
Then by \eqref{cancelation2} 
 \[
 \left|\int k(\tau,0,x)  g(\tau,x)   dx \right| \le c [\rho (1+|\tau|^2)]^{-1}   
\] 
and hence 
\[  |Tg(0)|+ |Tg(x_1)| \le c \rho^{-1}  \]  
independent of $f$. It tends to $0$ as  $\rho\to \infty $. As a consequence it suffices to prove the key inequality \eqref{mainstrange} 
for $f$ which satisfies $f(t,0)=0$, which we assume from now on. 

\noindent{\bf Step 2:} 
Let $h(t,x) = \eta(2(x-x_1)) f(t,x_1) $.
Then, using the cancellation condition, for $x=0$ and $x=x_1$,  
\[ \left|  \int_{\R^{n-1}}   k(\tau,0,x-y) h(\tau,y) dy \right|\le c |f(t,x_1)| \] 
for $|\tau|\le 2$.  Let $\chi= \chi_{[-1,1]}(t)$ and 
$f_2 = f  - \chi h$. It vanishes at $x=0$ and 
at $x=x_1$ for $|t|\le 1$ . Then  
\[
\begin{split} 
 \left| \int k(\tau,-y) f_2(\tau,y) dy \right| \le & 
c \int (|\tau|+|y|)^{-n} |y|^s dx \sup_{z}  \frac{|f(\tau,z)|}{|z|^s}
\\ \le & c |\tau|^{s-1} \sup_{x\ne y}  \frac{|f(\tau,x)-f(\tau,y)|}{|x-y|^s}
\end{split} 
\] 
which we use for $|t|\le 1$. The same bound holds at $x=x_1$. 
Together 
\[ \left| (T (\chi  f) (0,x_1)- (T\chi  f)(0,0)  \right| 
\le  c \sup_{|t|\le 1} \sup_{x\ne y}  \frac{|f(t,x)-f(t,y)|}{|x-y|^s}.
\]  
We turn to $|\tau|\ge 1$ and estimate 
 \[\begin{split} 
 \left| \int_{\R^{n-1}}  (k(\tau,0,-y)-k(\tau,0,x_0-y))  f(t,y) dy \right|  \hspace{-2cm} & \\
 \le & 
\int_{\R^{n-1}} |y|^s  (|\tau|+|y|)^{-n-1}  dy  \sup_{x\ne y } \frac{|f(\tau,x)-f(\tau,y)|}{|x-y|^s} 
\\ \le & c   |\tau|^{-2+s} \Vert f(t,.)  \Vert_{\dot C^s(\R^{n-1})}.  
\end{split} 
\] 
We integrate that estimate over $\R \backslash (-1,1)$ and arrive at
\eqref{mainstrange}.  The proof extends to $f\in L^\infty \dot C^s$,
and this allows to extend the assertion to the weak (distributional)
closure of $\dot C^s$ in that space, which is the whole
space.  
\end{proof}

We may weaken the assumption at least in the scalar case. Consider 
\[  \partial_{x_i} a^{ij}(t) \partial_{x_j} u = \partial_i F^i \]
with measurable uniformly bounded and elliptic coefficients
$a^{ij}$. We use the index $0$ for the coordinates corresponding to
$t$. We use the same convention for the Green's function as for the
kernel above.

\begin{lemma}\label{green}  The Green's function satisfies 
\[  |\partial_x^\alpha  \partial_t^l \partial_\tau^k    g(\tau,t,x) |
\le c (|x|+|t-\tau|)^{2-n-|\alpha|-l-k} \] 
for $k,l\le 1$ and all multiindices $\alpha$.  
\end{lemma} 

\begin{proof} There exists a unique Green's function $G$ which satisfies
\[ G((x,t);(y,\tau)) = g(\tau,t,x-y)= g(t,\tau,y-x) \] 
for some function $g$, 
\[   |g(t,s,x)| \le c (|x|+|s-t|)^{2-n} \] 
if $ n \ge 3$, see Gr\"uter and Widman \cite{MR657523}. Moreover we show that 
with $4r^2= |x|^2+ |t-s|^2$
\[  r^{-n}   \int_{B_r(x,t)} |\nabla_{\tau,x} g(\tau,t,y)|^2 dy ds \le c r^{1-n}.   \] 
By a scaling argument it suffices to consider a ball of radius $1$
around $(t_0,x_0)$ with $|t_0-s|^2 + |x_0-y|^2= 4$, and hence to bound
\[ |\partial^\alpha \partial_t^l u(0,0) | \le c_{\alpha}  \Vert u \Vert_{L^2(B_1(0,0))} \] 
for $l\le 1$ and a solution $u$ to the homogeneous problem in $B_2$. 
Recursive $L^2$ estimates  imply the estimates for 
\[  \partial^\alpha_x \partial_t^l  u  \] 
in $L^2$ for $l \le 1$.  This implies pointwise estimates 
for $\partial^\alpha_x u$ in terms of the $L^2$ norm. We rewrite  
\[ \partial_t g^{0j} \partial_{j} \partial^\alpha u 
= -  g^{i0} \partial_t\partial_i  \partial^\alpha u 
-  g^{ij} \partial^2_{ij} \partial^\alpha u. \] 
The second term on the right hand side is bounded. The first term is in 
$L^2_t L^\infty_x$. Thus 
\[ a^{0j} \partial_j \partial^\alpha u \] 
is bounded and hence $\partial_t \partial^\alpha g$ is bounded in terms of the $L^2$ norm.  The Green's function is symmetric, and the remaining estimate for 
$\partial_t \partial_\tau \partial_x^\alpha g$ follows by repeating the previous arguments.  
\end{proof}

\begin{prop}\label{jump} Under these assumptions Proposition \ref{prop:strange} 
holds. 
\end{prop} 
\begin{proof} 
It suffices to check that the kernel bounds of Lemma \ref{green} are
sufficiently strong for the proof of Proposition \ref{prop:strange}.
Only the cancellation condition is not obvious. As above the
cancellation condition is immediate for $\partial^2_{x_i x_j}
g(\tau,t,x) $ and for $\partial^2_{x_i t} g(\tau,x)$ and
$\partial^2_{\tau x_i}g(\tau,t,x)$ since then the kernel contains a
derivative in $x$ direction.  Only $\partial_t \partial_s g(s,t,x)$
requires additional considerations.

The Green's function is a solution to the homogeneous problem away from the diagonal.  Let 
\[ u(t,x) = \partial_\tau g(\tau, t, x). \]
It satisfies 
\[  |u(t,x)| \le c (|t-\tau|+ |x|)^{1-n} \] 
and 
\[ |\partial_t u(t,x) |\le c (|t-\tau|+ |x|)^{-n}. \]
Moreover  it is a solution to the homogeneous equation away from $t=\tau$ and $x=y$ hence   
  \begin{equation}  \partial_t (a^{00} \partial_t u) = -
 \partial_{\alpha} \partial_{t} a^{0\alpha} u 
-  \sum_{\alpha,\beta  \ge 1} \partial_\alpha a^{\alpha \beta} \partial_{\beta} u 
- \sum_{\alpha\ge 1} \partial_\alpha a^{\alpha,0} \partial_t u \end{equation} 
Let $R >0$. We want to prove  for $t \ne \tau $ that 
\[  \left|\int_{B_R} u_t(t,x) dx \right| \le c R^{-1} \]  
which is trivial for $R\le |t-\tau| $. Let $t>\tau+R$ without loss of generality. Then, by the previous formula and an application of the divergence theorem   
\[ \begin{split} \left| \int_{B_R} a^{00} u_t(t,x)\, dx \right| 
\le & \left| \int_{B_R} \partial_{\alpha} a^{0\alpha} u\,  dx \right| 
+ \left| \int_{\partial B_R}\int_t^\infty  (R+ |s-\tau|)^{-n}\,  ds \right| 
\\ \le & C R^{-1}. 
\end{split} 
\] 
\end{proof} 

We conclude this section with some existence and uniqueness statements.

\begin{lemma}\label{solrn}  Let $ F^i_k \in C^s$ and $f \in L^{n/s}$
be supported in the unit ball. Then there exists a unique solution 
$u$ to 
\[ \partial_i a^{ij}_{kl}(x_n) \partial_j u^l = \partial_i F^i_k + f_k \] 
which satisfies 
\[ \Vert Du \Vert_{\dot C^s(\R^n)} \le c \left( \Vert F \Vert_{\dot C^s(\R^n)} 
+ \Vert f \Vert_{L^{n/s}} \right) \] 
and 
\[  u \to 0 \qquad \text{ as } x\to \infty \] 
if $ n\ge 3$ or 
\[ u(0) = 0 \]
and 
\[  |u| \le c ( 1+ |x|^{\varepsilon}) \]
if $n=2$. Similarly, if $a^{ij}(x_n)$ is bounded, measurable and uniformly elliptic, 
$F^i \in C^s_*$ with compact support and $ f \in L^{n/s}$ with support 
in $B_1(0)$ 
then there is a unique solution $u$ to 
\[ \partial_i a^{ij} \partial_j u = \partial_i F^i + f \] 
which satisfies 
\[  \Vert Du \Vert_{C^s_*} \le c \left( \Vert F \Vert_{C^s_*}  + \Vert f \Vert_{L^\infty} \right) \]
and  
\[  u \to 0 \qquad \text{ as } x\to \infty \] 
if $ n\ge 3$ or 
\[ u(0) = 0 \]
and 
\[  |u| \le c ( 1+ |x|^{\varepsilon}) \]
if $n=2$.   
\end{lemma} 
\begin{proof} The convolution with the fundamental solution gives a function 
which satisfies the homogeneous estimates. The solution is unique up to the addition of an affine function. The condition fixes this affine function. 
The estimates in the inhomogeneous norm are a consequence of the bounds 
of the fundamental solution. The same arguments apply to the second part. 
\end{proof}  

\section{Partial holomorphy for linear systems} 
\label{sec:linan} 

We consider the elliptic system of equations 
\begin{equation} \label{linear} 
\partial_i a^{ij}_{kl}  \partial_j u^l = \partial_i F^i_k + f_k 
\end{equation}
on  $B_2(0) \subset \R^n$ under the assumption that the constant coefficients are elliptic and    $ u \in C^{1,s}(\overline{B_1(0)})$ for some $s>0$.

We will consider partially analytic functions $F$ and $f$ which are
given as partially holomorphic functions in a partially complex
domain. Given $\delta>0$ we define the complexified unit ball
\[ B_{\delta}  = \Big\{ (x'+iy', x'') :  |y'| \le \delta (1-|x|^2)_+^3 \Big\}. \] 

\begin{defn} 
Let $\delta>0$ and $0<s<1$. 
We define the  norms 
\[ \Vert F \Vert_{C^s_\delta} = \sup_{\theta \in \R^{n_1}, |\theta| \le \delta}    \Vert  F (x'+i\theta (1-|x|^2)^3_+ ,x'') \Vert_{C^s(B_2(0))}, \]
\[ \Vert F \Vert_{C^s_{*,\delta}} = \sup_{\theta \in \R^{n_1}, |\theta| \le \delta}    \Vert  F (x'+i\theta (1-|x|^2)^3_+ ,x'') \Vert_{C^s_*(B_2(0))}, \]
and 
\[ \Vert f \Vert_{L^\infty_\delta } 
=   \Vert f  \Vert_{L^\infty(B_\delta)\cap L^\infty(B_2(0))}.     \] 
The corresponding function spaces are the spaces of functions for which these 
norms are finite, and which are  holomorphic in 
$z'=x'+iy'$.  
\end{defn} 

Holomorphy is equivalent to the Cauchy-Riemann equations. The Cauchy integral 
formula implies estimates for derivatives. 
It will be useful later on that we allow $x$ in the ball of radius $2$.
Let $F \in C^s_\delta$. We define 
\[  F^0(x'+iy',x'') := F(x',x''). \]  

\begin{prop}
\label{linearhol} 
 Suppose that $n\ge 2$, and that  the coefficients $a^{ij}$ are elliptic. There exists $\delta_0>0$ so that the following holds. Suppose that  $0<\delta \le  \delta_0$, 
\[ F \in C^s_{\delta} ,   f \in L^\infty_\delta \]
and  that $u^0 \in C^{1,s}(B_2)$ satisfies 
\[   \partial_i a^{ij}_{kl} \partial_j u^{0,l} = \partial_i F^i_k + f_k.\] 
Then there exists a unique partially holomorphic extension to
$B_{\delta}$ with 
\[  \Vert u-u^0(x)  \Vert_{C^{1,s}} \le c\left(  \delta \Vert u^0 \Vert_{C^{1,s}} + 
 \Vert F - F^0 \Vert_{C^s_\delta} + \Vert f-f^0 \Vert_{L^\infty_{\delta}}\right). \] 
\end{prop} 

\begin{proof} 
We consider only scalar equations of the type 
\[ \partial_l a^{lk} \partial_k u = f \] 
to simplify the exposition. 
The case of systems requires not more than obvious modifications. 
We treat weak derivatives on a formal level. This can be justified by testing 
by a function in $\R^{n_2}$, integrating, and checking that the derivatives 
with respect to the first $n_1$ derivatives always exist.

If $u$ with $Du \in C^s_{\delta_1}$ is a partially holomorphic solution then 
the function (with $|\theta| \le \delta$)
\[ u^{\theta}(x) =   u(x'+i\theta (1-|x|^2)^3_+, x'')  \] 
satisfies 
\begin{equation} \label{utheta} L^\theta u^\theta := \partial_{X_i}  a^{ij}  
\partial_{X_j}  u^\theta
=    f_\theta \end{equation} 
where   (with $\phi= (1-|x|^2)^3_+$)
\begin{equation} \label{Xj}  \partial_{X_j} = \partial_{x_j} - \frac{i \partial_{x_j} \phi}{1+i\theta\cdot  \nabla \phi} \theta_l   \partial_{x_l}. 
\end{equation}  

 Equation \eqref{utheta} is elliptic provided $\delta_0 $ is
 sufficiently small.  It will be useful to consider
\[ v^\theta = u^\theta(x) - u(x)  \] 
which we extend by $0$ outside the unit ball. 
The function $v^\theta$ has compact support, and it is  a solution to 
\begin{equation} \label{vtheta}   
 \partial_i a^{ij}\partial_j v^{\theta}  
= \partial_i (a^{ij}-a^{ij}_{\theta})\partial_j v^\theta +  
\partial_i (F^i_\theta - F^i)+ (f_\theta-f)  - \partial_i (a^{ij}_\theta- a^{ij})  \partial_j u \end{equation}
on $\R^n$     with $v^\theta$ and the  right hand side  supported in $B_1(0)$. 
A simple fixed point  argument shows that this equation has a unique solution 
which satisfies 
\[ \Vert \nabla v^\theta \Vert_{C^{s}} \le c \Big( \Vert F^i_\theta - F^u
\Vert_{C^s} + \Vert f_\theta - f \Vert_{L^{\infty}} + |\theta| \Vert
Du \Vert_{C^s} \Big)\] 
as well as  
\begin{equation}\label{normal1}  \nabla v^\theta (x)\to 0 \qquad \text{ as }  
x \to
\infty\end{equation} 
and 
\begin{equation} \label{normal2} 
 v^\theta(x_0)=0 \qquad \text{ at a chosen point } x_0. \end{equation} 
This yields the estimate of Proposition \ref{linearhol} and it remains
to verify existence of a partially holomorphic solution.

For the existence argument we consider $\theta \in \R^{n_1}$ with
$|\theta| < \delta$ and define $v^\theta$ (
and hence $u^\theta= v^\theta + u$) as solution to \eqref{vtheta} with
the normalizing conditions \eqref{normal1} and \eqref{normal2} where
$x_0$ is a point outside the ball of radius $2$ .  We obtain a family of
solution to \eqref{utheta}.  Since we construct the solution via a
fixed point argument respectively via an implicit function theorem
$v^\theta$ (which we normalize by $v^\theta(x_0)=0$ for a point $x_0$
with $|x_0|=2$, and $\nabla v^\theta \to 0 $ as $x \to \infty$) and
hence $u^\theta$ depends differentiably on $\theta$.

The claimed bound for $u^\theta$ is an immediate consequence of
elliptic regularity estimates. Partial holomorphy however requires an argument. 
In the next step we connect derivatives with respect to $\theta$ to derivatives 
with respect to $x$. This connection is obvious if  the solution $u$ is partially holomorphic: Recall that 
\[ u^{\theta}(x)  =  u( x+ i\theta \phi(x))  \] 
Let $1 \le L \le n_1$ and 
\[ \dot u  := \partial_{\theta_L} u^\theta  = 
  i \phi  \partial_{z_L} u 
= i \phi \partial_{X_L} u^\theta 
 \] 
where the second and third identity assume  holomorphy  and 
\[ 
\tilde u :=    i\phi  \partial_{X_L} u^\theta.
\] 
We suppress $\theta$ in the notation of $\tilde u$ and $\dot u$.
The Cauchy Riemann equations are equivalent to 
\begin{equation} \label{identity}  
\dot u = \tilde u,  
 \end{equation}  
which we have to verify without assuming holomorphy. 
  An easy calculation shows that (see \eqref{Xj} for a definition of the vector fields) 
\[ \partial_{\theta_L} \partial_{X_j} \psi(x)   = -\frac{i \phi_j}{1+i\theta \nabla \phi} \partial_{X_L}\psi(x)  \] 
hence   differentiating the differential equation gives  
\begin{equation}\label{dotu} 
   L^\theta \dot u =   
\partial_{X_k}  \Big(a^{km}(\frac{i\partial_m\phi}{1+i\theta \nabla \phi} ) \partial_{X_L} u\Big)  
+ (\frac{i \partial_k \phi}{1+i\theta \nabla \phi} )\partial_{X_L} \Big( a^{kl} \partial_{X_l} u\Big)  
+ \partial_{\theta_L} f^\theta  
\end{equation}

One  easily checks that 
\[ \partial_{X_j}  \phi = \frac{\partial_{x_j}
  \phi}{1+i\theta \nabla \phi}, 
\]
and 
\[ \left[  \partial_{X_j},  
  \partial_{X_l} \right] 
= 0 . \]

As a consequence many of the expressions commute when we apply $L^\theta$ to $\tilde u$:
\begin{equation}\label{tildeu} 
\begin{split} 
L^\theta \tilde u = &  
 \partial_{X_j} a^{jk} \partial_{X_k} (i\phi \partial_{X_L} u)  
\\ = & \partial_{X_j} a^{jk}  \frac{i \partial_{x_k} \phi}{1+i\theta \nabla \phi} 
\partial_L u  +   \frac{i \partial_{x_j} \phi}{1+i\theta \nabla \phi}\partial_{X_L}  a^{jk} \partial_{X_k}u  
 + i\phi \partial_{\theta_L} f^\theta . 
\end{split} 
\end{equation}  
Solutions to equation \eqref{dotu} with compact support are unique. Since 
$\dot u - \tilde u$ is a solution with compact support  to the homogeneous equation and since $\partial_{\theta_L} f^\theta  = i \phi \partial_{X_L} f^\theta$ 
for holomorphic functions $f$ we obtain 
$  \dot u = \tilde u$.
\end{proof} 

There are only minor changes in the scalar case with $n_2=1$ and for  
coefficients depending measurably on $x_n$. Consider 
\begin{equation}\label{scalarhol}   \partial_{i} a^{ij}(x_n) \partial_j u = \partial_j F^j + f \end{equation}  
where the coefficients $a^{ij}$ are uniformly elliptic and measurable.

\begin{prop}
\label{linearhol2} 
Proposition \ref{linearhol} holds for \eqref{scalarhol} if we replace 
the function spaces by $C^s_{*,\delta}$ (with the obvious definition).
\end{prop}

\section{Partial analyticity of solutions to the nonlinear equation} 
\label{sec:nonlinan} 
There are two steps: First we prove that the Lipschitz solutions have 
H\"older continuous derivatives. In a second step we characterize the 
solution as the fixed point of a fixed point problem 
in a complexioned set, where we use at each step of the iteration 
the results of the previous section. 

\subsection{H\"older regularity of derivatives in Theorem \ref{theorem1}} 

Let $u$ be a Lipschitz continuous weak solution to the  elliptic problem  
\begin{equation} \label{quasi}  \partial_i F^i_k(x,u,Du) = f_k(x,u,Du) \qquad \text{ in } B_2(0) \end{equation}  
under the assumptions of Theorem \ref{theorem1}. 

\begin{prop} \label{hoelder} Under the assumptions of Theorem \ref{theorem1} there exists $s>0$  so that $u \in C^{1,s}(B_{3/2}(0))$. 
\end{prop} 

\begin{proof} 
The argument could be iterated to yield (partial) smoothness. This we do not pursue, but we will prove partial analyticity in this section. 
We rewrite the differential equation in terms $v= u - \bar u - b\cdot x $.
It satisfies an equation of the same type, but with 
\[ \Vert v \Vert_{C^{0,1}} < \varepsilon.  \]
Adding a constant if necessary we  assume without loss of generality 
\[ F^i(0,0,0) = 0. \] 
We rewrite the equation \eqref{quasi} as 
\[  \partial_i a^{ij}_{kl} \partial_j v^l = \partial_i G^i_k(x,v,Dv) + f_k(x,v,Dv) \] 
where by the assumptions of Theorem \ref{theorem1} 
\[ \Vert f_k (x,u(x), Du(x)) \Vert_{L^\infty} + \Vert D_{u,p} f_k(x,u(x),Du(x)) \Vert_{L^\infty} \le \rho \] 
\[  G^i(0,0,0) = 0 \qquad   |G^i_k(x,u,p_2) - G^i_k(x,u,p_1)| \le \varepsilon |p_2-p_1|. \]
 Let $u_1$ be the solution to 
\[ \partial_i a^{ij}_{kl} \partial_j u^l_1 = f_k(x,u,Du). \]  
For $\sigma>0$ there exists a solution which satisfies 
\[   \Vert D u_1  \Vert_{\dot C^\sigma} \le c \rho  \] 
Then  $u_2= u-u_1$ satisfies 
\[ \partial_i a^{ij} \partial_j u_2^l=  \partial_i G^i_k(x,u_1+u_2,Du_1+Du_2).\] 
Let $h \in \R^n$ be small and $v_h= u_2(x+h)-u_2(x)$. Then 
\[  \partial_i a^{ij} \partial_j v_h = \partial_i \left[ \int_0^1  \frac{\partial G^i_k}{\partial P^j_l}(x,u,Du_1+Du_2 + \lambda (Dv^h))d\lambda \Big] \partial_j v^l_h + H^i\right] \] 
with 
\[ 
\begin{split} 
  H^i =  &   G^i_k(x+h, u(x+h), Du_1(x+h)+ Du_2(x+h))
\\ &  - G^i_k(x,u(x) , Du_1(x) + Du_2(x+h)).
\end{split} 
\] 
Under the assumptions of the theorem
\[   \Vert H \Vert_{C^s} \le  c \rho. \]
 Let $\eta$ be a cutoff function and $w= \eta v_h$. Then by the Calder\'on-Zygmund estimate  
\[ \Vert Dw \Vert_{L^p} \le c \rho |h|^s + \varepsilon \Vert Dw \Vert_{L^p} \] 
and hence 
\[ \Vert Dw \Vert_{L^p} \le c (\rho+\varepsilon)  |h|^s. \]
By Morrey's estimate 
\[ \Vert w \Vert_{C^{1-\frac{n}p}} \le c(\rho+\varepsilon)  |h|^s \] 
and hence
\[   |u_2(x+h)-2u_2(x)+u_2(x-h)| \le c(\rho+\varepsilon) (|h|^{1-\frac{n}p}+ h^\sigma)) \] 
with an exponent $s$ if $p$ is sufficiently large, $h$ is small and $x$ is in 
the interior. This bound implies 
\[    \Vert Du_2 \Vert_{C^{s-\frac{n}p}(B_{3/2}(0)} \le c (\rho+\varepsilon).  \]
This completes the proof. 
\end{proof}

Let $u$ be a Lipschitz continuous weak solution to the elliptic problem  
\begin{equation} \label{quasi2}  \partial_i F^i_k(x,u,Du) = f_k(x,u,Du) \end{equation}  
under the assumptions of Theorem \ref{theorem1a}. This requires only minor   modifications and we state the result.

\begin{prop}\label{einsa*}  Under the assumptions of Theorem \ref{theorem1a} there exists $s>0$ 
so that $Du \in C^{s}_{*}(B_{3/2})$. 
\end{prop}

\subsection{The nonlinear equation: Analyticity} 

In this subsection we prove Theorem \ref{theorem1} and Theorem \ref{theorem1a}.
The arguments are again very similar. 

\begin{proof}[Proof of Theorem \ref{theorem1}] We work under the assumptions of Theorem \ref{theorem1}, specifying several small parameters along the way. 
By Proposition \ref{hoelder} the weak solution has  H\"older continuous derivatives in  $B_{3/2}(0)$.  We subtract  $b  x + v(0)$ to reduce the problem to the special situation  $b=0$ and  $v(0) = 0 $. Let $\eta \in C^\infty_0(B_{3/2}(0))$ be identically $1$ in $B_1(0)$.  Then  
\[ \tilde u  = \eta u \] 
 satisfies
\[ \partial_i a^{ij} \partial_j \tilde u =\partial_i \tilde F^i + \tilde f \] 
where 
\[ \tilde F^i = \eta F^i(x,(1-\eta)u+ v ,D((1-\eta) u) + D\tilde u) - a^{ij} \partial_j \tilde u   
+ a^{ij} (\partial_j \eta)u   \] 
and 
\[ \tilde f = (\partial_i \eta)\Big[F^i(x,u,Du)-a^{ij} \partial_j u\Big]  + \eta 
f(x, (1-\eta) u + \tilde u , D(1-\eta ) u+D\tilde u ). \] 

 We characterize  $\tilde u $ as a fixed point of the map which maps 
$v$ to the solution to 
\begin{equation} \label{perturbation}    \partial_i a^{ij} \partial_j \tilde u = \partial_i \tilde F^i(x,v, Dv) +  \tilde f(x,v, Dv) \end{equation} 
where we suppress the dependence on $u$, which is trivial in $B_1(0)$. 
We normalize $\tilde u $ by choosing a point $x_0$ outside the ball with $\tilde u (x_0)=0$ 
and require $\nabla \tilde u  \to 0 $ as $x\to \infty$. 

For small $\delta$ and $\varepsilon$ to be chosen later we define
\[\begin{split}  X = & \Big\{ u \in C^{1,s}_\delta(U): \sup_{|\theta | \le \delta} \Vert u^\theta \Vert_{C^1(\R^n)} <2\varepsilon, 
\\ & \qquad \Vert u^\theta  \Vert_{C^{1,s}(\R^n)}\le R,  u(x_0)=0 , \nabla u(x) \to 0 \text{ as } x \to \infty \Big\}
\end{split} 
 \]
for some $R$ to be chosen later. 

As in  Proposition \ref{linearhol}

\[ \begin{split} \Vert Dv -D\tilde u^0 \Vert_{C^s_\delta} 
\le &  c \sup_{|\theta|\le \delta}  
\Big( \Vert \tilde F^\theta(x,\tilde u^\theta,D\tilde u^\theta)-F(x,\tilde u^0, D\tilde u^0) \Vert_{C^s} 
\\ & + \Vert \tilde f^\theta(x,\tilde u^\theta,D\tilde u^\theta)-f(x,\tilde u^0, D\tilde u^0) \Vert_{sup}\Big). 
\end{split} 
\] 

By the triangle inequality
\[ 
\begin{split} 
\Vert & F^{\theta}(x'+i\theta (1-|x|^2)_+^3,x'', u^\theta , Du^\theta) 
- F(x,u,Du) 
\Vert_{C^s} 
\\ & \le \Vert F^{\theta}(x'+i\theta (1-|x|^2)_+^3,x'', u^\theta , Du^\theta) 
- F(x,u^\theta,Du^\theta)  \Vert_{C^s} 
\\ & \qquad + \Vert F(x,u^\theta,Du^\theta) - F(x,u,Du) \Vert_{C^s}. 
\end{split} 
\] 
A straight forward estimate gives for $\rho > \varepsilon$
\[ \Vert F^{\theta}(x'+i\theta (1-|x|^2)_+^3,x'', u^\theta , Du^\theta) 
- F(x,u^\theta,Du^\theta)  \Vert_{C^s} 
\le  c (\theta + (\rho +\varepsilon) \Vert Du^\theta \Vert_{C^s} ) 
\] 
and 
\[  \Vert F(x,u^\theta,Du^\theta) - F(x,u,Du) \Vert_{C^s} 
\le    (\rho+\varepsilon ) \Vert Du^\theta -Du \Vert_{C^s} \]  
which implies 
\[ \Vert Du \Vert_{C^s_*} \le c \Vert u^0 \Vert_{C^{1,s}} + c (\rho + \varepsilon) R \]
We choose $R= 2 \Vert u^0 \Vert_{C^{1,s}(B_{3/2})} $.  
The similar $L^2$ estimate gives 
\[ \Vert D(u-u^0)  \Vert_{L^2} \le c (\theta + (\rho+\varepsilon) 
\Vert D(u-u^0) \Vert_{L^2} \] 
and hence 
\[ \Vert D(u-u_0) \Vert_{L^2} \le c |\theta|  \] 

\begin{lem} 
Let $|B_1|$ be the volume of the unit ball. Then 
\[ \Vert f \Vert_{sup} \le   2  |B_1|^{-\frac{s}{2s+n}}  \Vert f \Vert_{L^2}^{\frac{s}{s+n/2} } 
   \Vert f \Vert_{C^{\frac{n/2}{s+n/2}}}.   \] 
 \end{lem} 

\begin{proof} 
Clearly
\[ |f(x)| \le r^{s} \Vert f \Vert_{\dot C^s} +  |B_1|^{-1}  r^{-n} \int_{B_r(X)}  |f| dx.   \] 
 By H\"older's inequality 
\[ |f(x)| \le r^s \Vert f \Vert_{\dot C^s} + |B_1|^{-1/2}  r^{-n/2} \Vert f \Vert_{L^2} .\] 
We optimize with respect to $r$ and arrive at the assertion of the lemma.
\end{proof}

Thus 
\[ \Vert D(u-u^0) \Vert_{sup} \le c |\theta|^{\frac{s}{s+n/2}} 
 \Vert u^0 \Vert_{C^{1,s}}^{\frac{n/2}{s+n/2}}.  \]
Choosing $\theta$ sufficiently small implies 
\[ \Vert D(u-u^0) \Vert_{sup} \le \varepsilon. \] 
  Then the fixed point map maps a
ball in $X$ to itself. It is easy to see that it is a contraction in
the norm
\[ \sup_{\theta} \Vert u(x'+i\theta (1-|x|^2)_+^3,x'') \Vert_{H^1}. \] 
\end{proof}

\begin{proof}[Proof of Theorem \ref{theorem1a}] 
There are only minor changes for Theorem \ref{theorem1a}. 
We define 
\[   \bar v(x) = u(x_n) +  \sum_{i=1}^{n-1} b^ix_i \]  
 and define
\[  v = u - \bar v \] 
Then 
\[ \Vert v \Vert_{C^{0,1}(U)} < \varepsilon \] 
at least after choosing an appropriate possibly smaller set $U$. 
By Proposition \ref{einsa*} for any multiindex in $\R^{n-1}$ 
\begin{equation}  
 \Vert \partial^{\alpha'} Dv \Vert_{sup} \le c_{\alpha'} \varepsilon. 
\end{equation} 
Moreover $v$ is a weak solution to 
\begin{equation}  \partial_i a^{ij}(x_n) \partial_j u = \partial_i G^i(x,u,Du) + \partial_n H(x,u,Du) \end{equation}  
with $H$ as in the previous section.
\end{proof}

\section{The change of coordinates} 
\label{sec:transform} 
\subsection{Theorem \ref{theorem1} implies Theorem \ref{theorem2}}

 Let $ u \in C^1(U, \R^N)$ satisfy the assumptions of Theorem \ref{theorem2}.
In particular we assume that  $(\partial_{x_{n_1+i}}
  u^k)_{1 \le i,k\le N}$ is invertible.  We define a diffeomorphism
  \[ \Xi: \R^{n} \ni x \to y= (x', u(x)) \in \R^{n}. \]
The Jacobi matrix is
\[  
D\Xi(x) =  \left( \begin{matrix} \frac{\partial y_i}{\partial x_j} \end{matrix} \right)_{1\le i,j \le n} 
 =  \left( \begin{matrix} 
  1 & 0 & \dots & 0 \\
  0 & 1 & \dots & 0 \\
  \vdots & \vdots & \ddots & \vdots \\
   u^1_{x_1}& u^1_{x_2} & \dots & u^1_{x_n} \\
    \vdots & \vdots & \ddots & \vdots \\
 u^N_{x_1} &  u^N_{x_2} &  \dots & u^N_{x_n} 
   \end{matrix} 
\right)  = \left( \begin{matrix}  1_{n_1,n_1} & 0_{n_1N} \\
                                  D_x' u    & D_{x}'' u \end{matrix} \right) .
\] 
We define $v(y) = x'' $ and $\Psi: y \to (y',v(y))$. Then 
\[ \psi \circ \Xi= \id \] 
and hence 
\[   \sum_{l=1}^N ( \partial_{y_{n_1+j}} v^l) (\partial_{x_{n_1+l}} u^k) = \delta_j^k\] 
and, if $1\le k \le n_1$ 
\[  \sum_{j=1}^N  ( \partial_{y_{n_1+j}} v^l) \partial_{x_{k}}    u^j = - \partial_{y_k} v^l  \] 
It is useful to write these formulas more compact as 
\[  D_{x}'' u  D_{y}'' v = 1 \qquad       D_{y}''v D_{x}' u = - D_{y}' v. \] 
Then 
\[ \int  F^i_k (x',x'',u,D_{x}'u,D_{x}''u)\partial_i \phi_k dx  = 
    \int G^l_k(y,v,Dv) \partial_{y_l} \phi_k dy \] 
with $p^j_l = \partial_{y_l}  v^j $ and 
\[ G_k(y,v,p)^i =  \left( \begin{matrix}    1 & 0 \\   (p'')^{-1} p'  & (p'')^{-1}\end{matrix} \right)_{il}   F_k^l(y',v,y'',-(p'')^{-1} p', (p'')^{-1}  )   \det p'' .   \]

\begin{lemma} The weak system
\[ \partial_i G^i_k(x,v,Dv) = g(x,v,Dv) \] 
is elliptic. 
\end{lemma} 
\begin{proof} 
We have to verify that 
\begin{equation} \label{checkell}  \left| (\frac{G^i_k}{\partial p^j_l}  \xi_i 
\xi_l \eta_k)_j \right| 
\ge \kappa |\xi|^2 |\eta|. \end{equation}
Let $A=(a_{ij})_{1\le ij \le m}$ be a square matrix. Then, from the expansion 
of the determinant, if $1 \le i,j,l  \le n$   
\[ \sum_l (\frac{\partial \det A}{\partial a_{il}})_{il}a_{lj} 
= \sum_i (\frac{\partial \det A}{\partial a_{li}})_{il}a_{lj} 
 = \det A \delta_{ij}   \] 
and thus  
\[
 \det D'' v D\Xi   = 
\left( \begin{matrix} 
\det D'' v \, 1_{n_1\times n_1} & 0_{n_1\times n_2} \\[3mm] 
-      \sum\limits_{l=1}^N \left( 
\frac{\partial \det D'' v }{\partial (\partial_{y_i}v^l)} \partial_{y_j} v^l \right)_{n_1<i \le N, 1\le j \le n_1}  
      &   \frac{\partial \det D'' v} {\partial (\partial_{y_j} v^{l-n_1}) } 
 \end{matrix} \right)_{jm} 
\]

The identity (for $1\le m,k,l\le N$) 
\[    \frac{ \partial^2 \det D''v}{\partial (\partial_{y_{n_1+m}} v^l)  \partial (\partial_{y_{n_1+l}} v^k)} 
+      \frac{ \partial^2 \det D''v}{\partial (\partial_{y_{n_1+k}} v^l)  \partial (\partial_{y_{n_1+l}} v^m)}  = 0 \] 
is a consequence of the expansion of the determinant. We claim that for 
$1\le k \le N, 1\le  m\le n$ 
\begin{equation} \label{claimI}   
\begin{split}  
0= & \frac{\partial}{\partial p^k_i}
\left( \begin{matrix}  
 \det p'' 1 & 0 \\[2mm] 
-  \left( \sum\limits_{l=n_1+1}^{n}  \frac{\partial \det p'' }{\partial p^l_m} p^l_j  \right)_{n_1 < m \le  n,  1\le j  \le 
n_1}  & \left(\frac{ \partial \det p'' }{\partial p^j_m}\right)_{n_1< j,m\le  n}  
\end{matrix} \right)_{mj}
\\[2mm] & + \frac{\partial}{\partial p^k_j}
\left( \begin{matrix}    \det p'' 1 & 0 \\[2mm]
 -  \left( \sum\limits_{l=n_1+1}^{n} \frac{\partial \det p'' }{\partial p^l_m} p^l_i  \right)_{n_1 < m \le n,1\le i \le 
n_1}   & \left(\frac{ \partial \det p'' }{\partial p^m_j}\right)_{n_1< j,m\le  n}  
\end{matrix} \right)_{mi}  
\end{split} 
\end{equation} 
As a consequence of the considerations above 
the sum in \eqref{claimI} vanishes for $n_1<i,j\le n$. 
The claim is trivial for $1\le i,j \le n_1$ since only the block matrix on the 
lower left  corner depends linearly on those coefficients. It remains to 
consider $1\le i \le n_1 < j \le n$.  The claim \eqref{claimI} follows now from 
\[\begin{split} 
 \frac{\partial (d\Xi_i^m)}{\partial(\partial_{y_j} v^k)} + \frac{\partial (
d \Xi_j^m)}{\partial(\partial_{y_i} v^k)} 
= & \frac{\partial \det D''v } {\partial(\partial_{y_j} v^k)} \delta_{im} 
-  \frac{\partial}{\partial_{(\partial_{y_i} v^k)}} \sum_{l=1}^{n_1} \frac{\partial \det D'' v}{\partial_{y_{y_j}} v^l} \partial_{y_m} v^l 
\\ = & \frac{\partial \det D''v}{\partial(\partial_{y_j} v^k)}\delta_{im}
- \frac{\partial \det D''v}{\partial(\partial_{y_j} v^k)}\delta_{im}
= 0. 
\end{split} 
\] 

Then 
\[
\begin{split} 
    b^{ij}_{kl} \xi_i \xi_j 
=  &  (\det D''v)^{-1}  (D\Xi^T \xi)_{i'} \xi_{j'} \frac{\partial F^{i'}_k(x,v,-(D''v)^{-1} D'v, (D''v)^{-1})}
{\partial (\partial_{y_j} v^l)} 
\\ = &  -  (D\Xi^T \xi)_{i'}(D\Xi^T \xi)_{j'}  a^{i'j'}_{kl'} 
\frac{\partial \det D''v}{\partial(\partial y_{l'} v^l) }     
\end{split} 
\] 
and \eqref{checkell} is an immediate consequence. 
\end{proof} 

To complete the proof we fix a point $x_0$ and define 
$u^r(x) = r^{-1}u((x-x_0)/r)$. The smallness assumptions are satisfied if we choose $r$ 
sufficiently small.

\subsection{Theorem \ref{theorem1a} implies Theorem \ref{theorem2a}  } 

Specializing the previous calculation we obtain the transformed problem 

\begin{equation}  \partial_i G^i(y,v,Dv) = \tilde g(y,v,Dv) \end{equation}  
where with $y= (y',y'')$ 
\[ G^i(y,v,Dv) =   v_n   F^i(y',v,y'', -(D_{t}v)^{-1}  D_{y'} v, (D_t v)^{-1}) \]  
if $1\le i < n$ and 
\[\begin{split}  G^n(y,v,Dv) = &  -v_i F^i(y',v,y'', -(D_{t}v)^{-1}  D_{y'} v, (D_t v)^{-1}) 
\\ & + F^n((y',v,y'', -(D_{t}v)^{-1}  D_{y'} v, (D_t v)^{-1}) ), 
\end{split} 
\] 
\[ g(y,v,Dv) =   v_n f(y',v,y'', -(D_{t}v)^{-1}  D_{y'} v, (D_t v)^{-1}). \]
By the previous section  the equation is elliptic.  The notion of ellipticity  in this scalar context simplifies to 
\[  a^{ij} \xi_i \xi_j \ge \kappa |\xi|^2 \]
with 
\[   a^{ij}(y,v,Dv) = \frac{\partial G^i(y,v,p)}{\partial p_j}. \] 
Now we change the notation, replace $n-1$ by $n$  and denote $y'$ by $x$ and $x_n=t$,  use the index $0$ 
for the time component and denote  $G$ again by $F$ and $g $ by $f$ so that 
the transformed equation becomes 
\[ \partial_i F^i(t,x,v,Dv) = f(t,x,v,Dv) \] 
where $F$ and $f$ depend analytically on $x$, $v$ and $Dv$ uniform in $t$. 
Hence they admit an extension into a partial complexification of the domain. 

 In order to apply Theorem \ref{theorem1a} we have to ensure 
that 
\[   v - \sum_{i=1}^n b_i x_i - v_0(t) \] 
has a small Lipschitz constant for some constants $b$ and a Lipschitz function 
$v_0$.

\begin{lemma} There exists $s>0$ so that 
\[ D_x v \in C^s, \qquad \partial_t v \in C^s_* \] 
\end{lemma} 

\begin{proof} 

 For $1\le k \le n$ one obtains (to do it rigorously one has to consider finite
differences)  
\[  \partial_i a^{ij}(t,x,v,Dv) \partial_j (\partial_k v)  = -\partial_i (\partial_u F \partial_k v + \partial_{x_k} F ) + \partial_{x_k} f   \] 
and, by the H\"older regularity result of De Giorgi, Nash and Moser 
\[  \partial_k u \in C^s \] 
for some $s>0$.  Now we apply Caccioppoli's inequality in balls $B_r(t,x)$ to get 
\[   \Vert D_{t,x} \partial_k u \Vert_{L^2(B_r)} \le c r^{\frac{n}2 +s -1} \] 
and hence 
\[ \Vert \partial_k F^i \Vert_{L^2(B_r)} \le c r^{\frac{n}2+s-1}. \] 
This in turn gives 
\[ \Vert \partial_t F^0 \Vert_{L^2(B_r)} \le c r^{\frac{n}2+s-1} \] 
and hence together 
\[ \Vert  F^0 - F^0_{B_r(x)} \Vert_{L^2} \le c r^{\frac{n}2+s}. \] 
In particular 
\[  |F^0_{B_{2^{-k}}(x)} - F^0_{B_{2^{1-k}}(x)} | \le c 2^{-ks} \] 
and 
\[ F^0 \in C^s \] 
By assumption we can solve 
\[ F^0(t,x,u,Du) = f \] 
for $\partial_t u $ and get 
\[ \partial_t u = \phi(t,x,u, D'u,f) \] 
where $\phi$ is analytic in all variables besides $t$. Thus 
\[  \partial_t u \in C^s_*.  \] 
\end{proof} 

Decreasing the domain if necessary we can ensure the assumption of 
Theorem \ref{theorem1a}. Thus $v$ is analytic with respect to $x$. 
The level surfaces at level $u_0$ are parametrized by 
\[   y' \to (y', v(y,u_0)). \]
Moreover the derivatives of $v$ are holomorphic with respect to $x$.
The claim on analyticity of the derivatives holds since  
\[ u_{x_n} = (v_t)^{-1} \] 
and, for $j < n$  
\[ u_{x_j} = - v_{y_n}^{-1} v_{y_j}. \]

\subsection{Water waves with surface tension} 

\begin{proof}[Proof of Theorem \ref{waterwave}, completion]
If $\kappa \ne 0$ we observe that 
at the free boundary 
\[ \partial_{y_1} \frac{u_{y_1}}{\sqrt{1+u_{y_1}^2}}  \] 
is bounded, hence $u_{y_2=0} \in C^{1,1}$. With a small modification of the proof 
of Theorem \ref{theorem1a} we obtain  $Du \in C^s_*$: we need an additional 
estimate for a linear problem in Lemma \ref{modbound} below.
\end{proof}

\begin{lemma}\label{modbound}  Let $(a^{ij}(x_n))_{1\le i,j \le n} $ 
and $(b^{ij})_{1\le i,j < n}$ 
be   bounded uniformly positive 
definite matrices. We consider the system 
\[    \sum_{i,j=1}^n \partial_i a^{ij} \partial_j u = \sum_{i=1}^{n} \partial_i f^i \] 
in $x_n>0$ with the boundary condition
\[     \sum_{i,j=1}^{n-1} \partial_i (b^{ij} \partial_j  u) =  \sum_{i=1}^{n-1}  \partial_i  g^i \qquad \text{ on } \{x_n=0\}.   \] 
Suppose that $g$ has a holomorphic extension and that $f$ has a partially 
 holomorphic extension. Then the same is true 
for $u$ and 
\[ \Vert Du \Vert_{C^s_{*,\delta}} \le c \Big(1+ \Vert f \Vert_{C^s_{*,\delta}} + \Vert  g \Vert_{C^s_\delta } \Big)  \]
provided $\delta$ is sufficiently small. 
 \end{lemma} 

A similar boundary value problem has been considered in \cite{MR2143526}.

\begin{proof}  We first obtain the interior bound  at the boundary
\[ \Vert Du|_{x_n=0} \Vert_{C^s_\delta} \le c (1+\Vert g  \Vert_{C^s_{\delta}} \] 
and then by the boundary analogue
\[ 
\Vert Du \Vert_{C^s_\delta} \le c \left( \Vert f \Vert_{C^s_\delta} + \Vert g \Vert_{s,\delta}\right),  
\]
compare the proof of the analogous statements Proposition \ref{linearhol} 
 and \ref{linearhol2}. 
\end{proof} 

There is no change in the setting of Theorem \ref{theorem2a} resp. 
for the water wave problem with surface tension. 

\printbibliography

\end{document}